\documentclass[11pt, a4paper]{siamart171218}

\usepackage{bm}
\usepackage{mathrsfs}
\usepackage{amssymb}
\usepackage{multirow}
\usepackage{caption}
\usepackage{colortbl}
\usepackage{lipsum}
\usepackage{amsfonts}
\usepackage{graphicx}
\usepackage{diagbox}
\usepackage{epstopdf}
\usepackage{algorithmic}
\usepackage{grffile}

\usepackage{auto-pst-pdf} 

\DeclareMathOperator{\divv}{\mathrm{div}}

\newcommand{\inner}[2]{(#1, #2 )}
\newcommand{\R}{\mathbb{R}}
\newcommand{\ssum}{\textstyle \sum}
\newcommand{\dx}{\, \textrm{d} x}
\newcommand{\ds}{\, \textrm{d} s}

\def\ep{\bm \varepsilon}
\def\sig{\bm \sigma}
\def\bU{\boldsymbol U}
\def\bV{\boldsymbol V}
\def\bP{\boldsymbol P}
\def\bff{\boldsymbol f}
\def\bu{\boldsymbol u}
\def\bv{\boldsymbol v}
\def\bp{\boldsymbol p}
\def\br{\boldsymbol r}
\def\bg{\boldsymbol g}
\def\bw{\boldsymbol w}

\def\bz{\boldsymbol z}
\def\bq{\boldsymbol q}
\def\bn{\boldsymbol n}
\def\balpha{\boldsymbol \alpha}
\def\Divv{\text{Div\,}}

\def\C{\Lambda}
\renewcommand{\div}{\operatorname{div}}

\definecolor{viol}{rgb}{0.75,0.15,0.95}
\definecolor{bluegreen}{rgb}{0.55,0.15,0.95}
\definecolor{bluegreen}{rgb}{0.25,0.68,0.78}
\definecolor{orange}{rgb}{0.0,0.0,0.0}

\usepackage{todonotes}
\newsiamremark{remark}{Remark}
\newsiamremark{hypothesis}{Hypothesis}
\crefname{hypothesis}{Hypothesis}{Hypotheses}
\newsiamthm{claim}{Claim}

\newcommand{\kam}[1]{\noindent \textcolor{blue}{#1}}

\headers{Robust approximation of generalized Biot-Brinkman problems}
{Q. Hong, J. Kraus, M. Kuchta, M. Lymbery, K.A. Mardal and M.E. Rognes}

\title{Robust approximation of generalized Biot-Brinkman problems
}

\author{Qingguo Hong\thanks{Department of Mathematics, Pennsylvania State University 
  (\email{huq11@psu.edu}).}
\and Johannes Kraus\thanks{Faculty of Mathematics, University of Duisburg-Essen 
  (\email{johannes.kraus@uni-due.de}).}
\and Miroslav Kuchta\thanks{Department of Numerical Analysis and Scientific Computing, Simula Research Labaratory
  (\email{miroslav@simula.no}).}
\and Maria Lymbery\thanks{Faculty of Mathematics, University of Duisburg-Essen 
  (\email{maria.lymbery@uni-due.de}).}
\and Kent-Andr\'e Mardal\thanks{Department of Mathematics, University of Oslo and Department for Numerical Analysis and Scientific Computing, Simula Research Laboratory 
  (\email{kent-and@simula.no}).}
\and Marie E. Rognes\thanks{Department of Numerical Analysis and Scientific Computing, Simula Research Labaratory and Department of Mathematics, University of Bergen 
  (\email{meg@simula.no}).}
}

\usepackage{amsopn}
\DeclareMathOperator{\diag}{diag}

\begin{document}

\maketitle

% REQUIRED
\begin{abstract}
  The generalized Biot-Brinkman equations describe the displacement,
  pressures and fluxes in an elastic medium permeated by multiple
  viscous fluid networks and can be used to study complex
  poromechanical interactions in geophysics, biophysics and other
  engineering sciences. These equations extend on the Biot and
  multiple-network poroelasticity equations on the one hand and
  Brinkman flow models on the other hand, and as such embody a range
  of singular perturbation problems in realistic parameter regimes. In
  this paper, we introduce, theoretically analyze and numerically
  investigate a class of three-field finite element formulations of
  the generalized Biot-Brinkman equations. By introducing appropriate
  norms, we demonstrate that the proposed finite element
  discretization, as well as an associated preconditioning strategy,
  is robust with respect to the relevant parameter regimes. The
  theoretical analysis is complemented by numerical examples.
\end{abstract}

% REQUIRED

%% REQUIRED
%\begin{AMS}
%
%\end{AMS}

\section{Introduction}
\label{sec:introduction}

The study of the mechanical response of fluid-filled porous media --
\emph{poromechanics} -- is essential in geophysics, biophysics and
civil engineering. Through a series of seminal works dating from 1941
and onwards~\cite{Biot1941general, Biot1955theory}, Biot introduced
governing equations for the dynamic behavior of a linearly elastic
solid matrix permeated by a viscous fluid with flow through the pore
network described by Darcy's law~\cite{darcy1856fontaines,
  whitaker1986flow}. Double-porosity models, extending upon Biot's
single fluid network to the case of two interacting networks, were
used to describe the motion of liquids in fissured rocks as early as
in the 1960s~\cite{Barenblatt1960basic, wilson1982theory,
  khaled1984theory}. Later, multiple-network poroelasticity equations
emerged in the context of reservoir
modelling~\cite{bai1993multiporosity} to describe elastic media
permeated by multiple networks characterised by different porosities,
permeabilities and/or interactions. Since the early 2000s,
poromechanics has been applied to model the
heart~\cite{nash2000computational, chabiniok2016multiphysics} as well
as the brain and central nervous
system~\cite{TullyVentikos2011cerebral, stoverud2016poro,
  Vardakis2016investigating, Chou2016afully,
  Guo_etal2018subject-specific}.

In addition to interactions between fluid networks, recently also the
viscous forces acting within each network have come to the
fore~\cite{barnafi2021mathematical, chapelle2014general,
  burtschell2019numerical, kedarasetti2021arterial}. At its core, the
effect of viscosity can be accounted for by replacing the Darcy
approximation in the poroelasticity model by a Brinkman
approximation~\cite{brinkman1949calculation,
  rajagopal2007hierarchy}. We here introduce multiple-network
poroelasticity models incorporating viscosity under the term
\emph{generalized Biot-Brinkman equations}. In a bounded domain
$\Omega \subset \mathbb R^d$, $d=1, 2, 3$ comprising $n$ fluid
networks, the generalized Biot-Brinkman equations read as follows:
find the displacement $\bu = \bu(x, t)$, fluid fluxes $\bv_i =
\bv_i(x, t)$ and corresponding (negative) fluid pressures $p_i =
p_i(x, t)$, for $i=1, \dots, n$ satisfying
\begin{subequations}
  \label{eq:mBB:t}
  \begin{align}
    -\divv \left (\sig(\bu) + \balpha \cdot \bp \bm I \right )  &= \bff, \label{eq:mBB:t1} \\
    -\nu_i\divv \ep(\bv_i) + \bv_i - K_i \nabla p_i&= \bm r_i,  \label{eq:mBB:t2} \\
    %- c_{i} \dot{p}_{i} - \alpha_i \divv \dot{\bu} - \divv \bv_i  - S_i(\bp) &=g_i, 
    - c_{i} \dot{p}_{i} - \bar \beta_i p_i + \alpha_i \divv \dot{\bu} + \divv \bv_i  
    + \bm \beta_i \cdot \bp &= g_i, 
    \label{eq:mBB:t3}
  \end{align}
\end{subequations}
over $\Omega \times (0, T)$ for $T > 0$, and where~\eqref{eq:mBB:t2}
and~\eqref{eq:mBB:t3} hold for $i = 1,
\ldots,n$. In~\eqref{eq:mBB:t1}, we have introduced the vector
notation $\bp = (p_1, \dots, p_n)$ and $\balpha = (\alpha_1, \dots,
\alpha_n)$, where $\alpha_i$ is the Biot-Willis coefficient associated
with network $i$. The elastic stress and strain tensors are:
\begin{equation}
  \sig(\bu) = 2\mu \ep(\bu) + \lambda \text{div}(\bu)\bm I, \quad \ep(\bu)  = \frac{1}{2}(\nabla \bu + (\nabla \bu)^T),
  %\label{constitutive_compatibility}
\end{equation}
respectively, and with Lam\'e parameters $\mu$ and
$\lambda$. Moreover, for each fluid network $i$, $\nu_i$ denotes the
fluid viscosity and $K_i$ is its hydraulic conductance
tensor. Furthermore, \eqref{eq:mBB:t3} is an equivalent formulation of
the standard multiple-network poroelasticity mass balance
equations~\cite{bai1993multiporosity, lee2019mixed,
  hong2019conservative} with transfer coefficients $\beta_{ij}$,
denoting $\bm \beta_i = (\beta_{i1}, \dots, \beta_{in})$ and $\bar
\beta_i = \sum_{j} \beta_{ij}$, when the fluid transfer into
network $i$ is given by
\begin{equation*}
  \ssum_{j = 1, j \not = i}^n \beta_{ij}( p_i - p_j).
\end{equation*}
The constants $c_i$ in~\eqref{eq:mBB:t3} denote the constrained
specific storage coefficients, see
e.g.~\cite{Showalter2010poroelastic} and the references
therein. Finally, the prescribed right hand side $\bff$ denotes body
forces, while $g_i$ denotes a fluid source and $\bm r_i$ represents an
external flux, both of the two latter in each network $i$. In the case
$n = 1$ and $\nu = 0$,~\eqref{eq:mBB:t} reduces to the Biot equations.

The generalized Biot-Brinkman problem~\eqref{eq:mBB:t} defines a
challenging system of PDEs to solve numerically. One reason for this
is the large number of material parameters, several of which give
rise to singular perturbation problems such as in the extreme cases of
(near) incompressibility ($\lambda \rightarrow \infty$) and
impermeability ($K_i \rightarrow 0$). Specifically, $\lambda \gg \mu$
is associated with numerical locking; if~\eqref{eq:mBB:t} is scaled by
$1/\lambda$, the elastic term of the equation reads $\divv \frac{2
  \mu}{\lambda} \ep(\bu) + \nabla \divv \bu = f$ which transforms from
an $H^1$ problem to an $H(\divv)$ problem as $\lambda$ tends to
infinity. Similar singular perturbation problems arise, now for the
flux variable $\bv_i$, as $\nu_i$ tends to zero. Furthermore, certain
parameter ranges of the storage coefficients and permeabilities ($K_i
\ll c_i)$ give rise to singular perturbation problems in the Darcy
sub-system, see e.g.~\cite{mardal2021accurate} and references
therein. Finally, we mention that large transfer coefficients $\beta_{ij}$
and/or small Biot-Willis coefficients $\alpha_i$ can lead to
strong coupling of the different subsystems and prevent direct
exploitation of each subsystem's properties.

In the case of vanishing viscosities ($\nu_i = 0, \forall i$) the
system~\eqref{eq:mBB:t} reduces to the multiple-network poroelasticity
(MPET) equations. Robust and conservative numerical approximations of
the MPET equations have been studied in the context of (near)
incompressibility~\cite{lee2019mixed} as well as other material
parameters~\cite{hong2019conservative, hong2020parameterM,
  hong2021framework}. Parameter-independent preconditioning and splitting
schemes as well as a-posteriori error analysis and adaptivity have also been
identified for the MPET
equations~\cite{hong2020parameter, hong2020parameterM, piersanti2021parameter,
eliseussen2021posteriori}. However, the generalized
Biot-Brinkman system has received little attention from the numerical
community. Therefore, the purpose of this paper is to identify and
analyze stable finite element approximation schemes and
preconditioning techniques for the time-discrete generalized
Biot-Brinkman systems, with particular focus on parameter robustness.

This paper is organized as follows. After introducing notation,
context and preliminaries in~\Cref{sec:preliminaries}, we prove that
the time-discrete generalized Biot-Brinkman system is well-posed in
appropriate function spaces in~\Cref{sec:wellposed}. We introduce a
fully discrete generalized Biot-Brinkman problem in
\Cref{sec:uni_stab_disc_model} and prove that the discrete
approximations satisfy a near optimal a-priori error estimate in
appropriate norms independently of material parameters. We also
propose a natural preconditioner. The theoretical analysis is
complemented by numerical experiments in \Cref{sec:numerics}.

\section{Preliminaries and notation}
\label{sec:preliminaries}

In this section of preliminaries, we give assumptions on the material
parameters, present a rescaling of a time-discrete generalized
Biot-Brinkman system and introduce parameter-weighted norms and
function spaces.

\subsection{Material parameters}

We assume that the elastic Lam\'e coefficients satisfy the standard
conditions $\mu > 0$ and $d \lambda + 2 \mu > 0$. The transfer
coefficients are such that $\beta_{ij} = \beta_{ji} \geq 0$ for $i
\not = j$ while $\beta_{ii} = 0$, and the specific storage
coefficients $c_i \geq 0$ for $i = 1, \dots, n$. The Biot-Willis
coefficients are bounded between zero and one by construction: $0 <
\alpha_i \leq 1$. We also assume that the hydraulic conductances $K_i
> 0$ for $i = 1, \dots, n$.  
Further, our focus will be on the case $\nu_i> 0$.
For spatially-varying material parameters,
we assume that each of the above conditions holds point-wise and that
each parameter field is uniformly bounded from above and below.

\subsection{Time discretization, rescaling and structure}

%% Considering first the (single network) Biot-Brinkman case ($n = 1$,
%% $\bm r = 0$) and then omitting the subscripts, we observe
%% that~\eqref{eq:mBB:t} reduce to the following system of equations:
%% find $\bu$, $\bv$ and $p$ such that
%% \begin{subequations}
%%   %\label{eq:biot-brinkman}
%%   \begin{align*}
%%     - \divv \left (\sig + \alpha p \bm I \right ) &= \bff, \\
%%     - \nu \divv \ep(\bv) + \bv - K \nabla p &= \bm 0 \label{eq:BB2}\\
%%     - c \dot{p} + \alpha \divv \dot{\bu} + \divv \bv  &=g, 
%%   \end{align*}
%% \end{subequations}
%% If $\nu = 0$, the velocity $\bv$ is given by the familiar Darcy's law
%% and the system reduces to the standard Biot equations. Similarly for
%% general $n$, if $\nu_i = 0$ for all $i$,~\eqref{eq:mBB:t} reduce to
%% the standard MPET equations, see e.g.~\cite{Bai_etal1993multi,
%%   lee2019mixed, hong2019conservative}.

Taking an implicit Euler time-discretization of~\eqref{eq:mBB:t} with
uniform timestep $\tau$, multiplying~\eqref{eq:mBB:t3} by $\tau$,
rearranging terms and removing the time-dependence from the notation,
we obtain the following problem structure to be solved over $\Omega$
at each time step: find the unknown displacement $\bu = \bu(x)$, fluid
fluxes $\bv_i = \bv_i(x)$ and corresponding (negative) fluid pressures
$p_i = p_i(x)$, for $i=1, \dots, n$ satisfying
\begin{subequations}
  \label{eq:brinkman_system}
  \begin{align*}
    -\divv \left (\sig(\bu) + \balpha \cdot \bp \bm I \right )  &= \bff,  \\
    - \nu_i \divv \ep(\bv_i) + \bv_i - K_i \nabla p_i &= \bm r_i,   \\
    - \left( c_{i} + \tau \bar \beta_i \right ) p_{i} + \alpha_i \divv \bu + \tau \divv \bv_i + \tau \bm \beta_i \cdot \bp &= \tau g_i. 
  \end{align*}
\end{subequations}
Multiplying by $\tau K_i^{-1}$ in the second equation(s) for the sake of symmetry gives
\begin{subequations}
  \label{eq:brinkman_system_scaled}
  \begin{align}
    - \divv \left (\sig(\bu) + \balpha \cdot \bp \bm I \right )  &= \bff,  \\
    - \nu_i \tau K_i^{-1}\divv \ep(\bv_i) +\tau K_i^{-1} \bv_i - \tau\nabla p_i &=\tau K_i^{-1} \bm r_i,   \\
    - \left( c_{i} + \tau \bar \beta_i \right ) p_{i} + \alpha_i \divv \bu + \tau \divv \bv_i + \tau \bm \beta_i \cdot \bp &= \tau g_i. 
  \end{align}
\end{subequations}

For the sake of readability, we define
\begin{equation}\label{eq:scale_params}
  s_i := c_i + \tau \bar \beta_i, \quad
  \gamma_i := \tau \nu_i K_i^{-1}
\end{equation}
recalling that $\bar \beta_i = \sum_{j} \beta_{ij}$ and $\beta_{ii} =
0$, and set
\begin{equation}
  R^{-1} := \max \{ (1 + \nu_1) \tau K_1^{-1}, \dots, (1 + \nu_n) \tau K_n^{-1} \}.
\end{equation}
Using this notation, we introduce four $n \times n$ parameter matrices
\begin{equation}
      \Lambda_1 = - \tau 
      \begin{pmatrix}
        0 & \beta_{12} & \dots & \beta_{1n}  \\
        \beta_{21} & 0 & \dots & \beta_{2n}  \\
        \vdots & \vdots & \ddots & \vdots  \\
        \beta_{n1} & \beta_{n2} & \dots & 0
      \end{pmatrix} \\
\end{equation}
and
\begin{equation}
  \Lambda_2 = \diag(s_1, s_2, \dots, s_n), \quad
  \Lambda_3 = \tau^{2} R I, \quad
  \Lambda_4 = \frac{1}{2\mu+\lambda} \boldsymbol{\alpha} \boldsymbol{\alpha}^T,
\end{equation}
before defining
\begin{equation}
  \Lambda = \sum_{i=1}^4 \Lambda_i .
\end{equation}
In the case $n = 1$, dropping the subscripts $i, j$ for readability
and with the newly introduced parameter notation, the operator
structure of the rescaled system~\eqref{eq:brinkman_system_scaled} is
\begin{equation}
  \begin{pmatrix}
    - \divv \sig & \bm 0 & - \alpha \nabla \\
    \bm 0 & - \gamma \divv \ep + \tau K^{-1} \bm I & - \tau  \nabla \\
    \alpha\divv & \tau\divv & - (\Lambda_1 + \Lambda_2)  \\
  \end{pmatrix}
  \begin{pmatrix}
    \bm u \\
    \bm v \\
    \bm p
  \end{pmatrix}
  =
  \begin{pmatrix}
    \bm f \\
    \bm r \\
    \bm g
  \end{pmatrix}
  \label{eq:bb:matrix}
\end{equation}
for $- (\Lambda_1 + \Lambda_2) =  c \bm I$, and $\bm p = p$ in the $n = 1$ case. The
same structure holds for $n > 2$ when denoting $\bm v^T = (\bm v_1^T, \bm
v_2^T, \dots, \bm v_n^T)$, $({\rm Div} \bm v)^T = ( \divv \bv_1, \dots, \divv
\bv_n)$.

By the assumption of symmetric transfer, i.e.~$\beta_{ij} =
\beta_{ji}$, $\Lambda_1$ and $\Lambda$ are symmetric. Moreover, as
$\Lambda_1 + \Lambda_2$ is weakly diagonally dominant and thus
symmetric positive semi-definite, $\Lambda_3$ is symmetric positive
definite, and $\Lambda_4$ is symmetric positive semi-definite, it
follows that $\Lambda$ is symmetric positive definite.

\subsection{Domain and boundary conditions}

Assume that $\Omega$ is open and bounded in $\R^d$, $d=2, 3$ with
Lipschitz boundary $\partial \Omega$. We consider the following
idealized boundary conditions for the theoretical analysis of the
time-discrete generalized Biot-Brinkman system~\eqref{eq:bb:matrix}
over $\Omega$. We assume that the displacement is prescribed (and
equal to zero for simplicity) on the entire boundary $\partial
\Omega$. Furthermore for each of the flux momentum equations we assume
datum on the normal flux $\bv_i \cdot \bn$ \emph{and} the tangential
part of the traction associated with the viscous term $\ep(\bv_i)
\cdot \bn$.  Combined, we thus set
\begin{equation}
  \begin{aligned}
    \label{eq:bcs}
    \bu(\bm x) &= {\bm 0} \quad {\bm x} \in \partial \Omega, \\
    \bv_i \cdot \bn (\bm x) = {\bm 0}, \,
      \bn \times \left(\ep(\bv_i)\cdot\bn\right) (\bm x)  &= {\bm 0} \quad {\bm x} \in \partial \Omega,
  \end{aligned}
\end{equation}
for $i = 1, \dots, n$. 

\subsection{Function spaces and norms}

We use standard notation for the Sobolev spaces $L^2(\Omega)$,
$H^1(\Omega)$ and $H(\divv, \Omega)$, and denote the
$L^2(\Omega)$-inner product and norm by $\inner{\cdot}{\cdot}$ and $\|
\cdot \|$, respectively. We let $L^2_0(\Omega)$ denote the space of
$L^2$ functions with zero mean. For a Banach space $U$, its dual space
is denoted $U'$ and the duality pairing between $U$ and $U'$ by
$\langle \cdot, \cdot \rangle_{U' \times U}$.

For the displacement, flux and pressure
spaces, we define
\begin{subequations}
  \label{eq:spaces}
  \begin{align}
    \bU &= \{\bu \in H^1(\Omega)^d : \bu = \boldsymbol 0 \text{ on }  \partial \Omega \},  \\
    %\bV_i &= \{\bv_i \in H^1(\Omega)^d  : \bv_i = \boldsymbol 0   \text{ on } \partial \Omega   \},
    \bV_i &= \{ \bv_i \in H^1(\Omega)^d  : \bv_i \cdot \bn = 0  \text{ on } \partial \Omega  \}, \\
    P_i &= L^2_0(\Omega),
    % \quad (\text{or } L^2_0(\Omega) \text{ iff } \Gamma_D = \partial \Omega),
  \end{align}
\end{subequations}
for $i = 1, \dots, n$, and subsequently define
\begin{equation}
  \bV = V_1 \times \cdots \times V_n, \quad
  \bP = P_1 \times \cdots \times P_n.  
\end{equation}
We also equip these spaces with the following parameter-weighted inner products
\begin{subequations}
  \label{norms}
  \begin{align}
    (\bu, \bw)_{\bU} & = (2\mu \ep(\bu), \ep (\bw)) + \lambda(\divv \bu, \divv \bw) ,
    \label{norms-u} \\
    (\bv, \bz)_{\bV} &= \sum_{i=1}^n \inner{\gamma_i \ep(\bv_i)}{\ep(\bz_i)} + \inner{\tau K_i^{-1} \bv_i}{\bz_i}
    + \inner{\Lambda^{-1} \tau^2 {\rm Div} \bv}{{\rm Div} \bz} ,
    \label{norms-v} \\
    (\bp, \bq)_{\bP} &= \inner{\Lambda \bp}{\bq}
    \label{norms-p}
  \end{align}
\end{subequations}
and denote the induced norms by $\|\cdot\|_{\bU}$, $\|\cdot\|_{\bV}$,
and $\|\cdot\|_{\bP}$, respectively. These are indeed inner products
and norms by the assumptions on the material parameters given and in
particular the symmetric positive-definiteness of $\Lambda$. 

%%%%%%%%%%%%%%%%%%%%%%%%%%%%%%%%%%%%%%%%%%%%%%%%%%%%%%%%%%%%%%%%%%%%%%%%%%%%%%%%%%%%
%%%%%%%%%%%%%%%%%%%%%%%%%%%%%%%%%%%%%%%%%%%%%%%%%%%%%%%%%%%%%%%%%%%%%%%%%%%%%%%%%%%%
%%%%%%%%%%%%%%%%%%%%%%%%%%%%%%%%%%%%%%%%%%%%%%%%%%%%%%%%%%%%%%%%%%%%%%%%%%%%%%%%%%%%

\section{Well-posedness of the Biot-Brinkman system}
\label{sec:wellposed}

\subsection{Abstract form and related results}

System~\eqref{eq:bb:matrix} is a special case of the abstract saddle-point problem
\begin{equation}
  \begin{pmatrix}
    A_1 & 0 & B_1^T \\
    0 & A_2 & B_2^T \\
    B_1 & B_2 & - A_3
  \end{pmatrix}
  \begin{pmatrix}
    \bu \\
    \bv \\
    \bp
  \end{pmatrix},
  \label{eq:two-fold}
\end{equation}
where $A_1 : \bm U \rightarrow \bm U'$, $A_2 : \bm V \rightarrow \bm V'$,
and $A_3 : \bm P \rightarrow \bm P'$ are symmetric and positive
(semi-)definite, and $B_1 : \bm U \rightarrow \bm P'$, $B_2 :\bm V \rightarrow
\bm P'$ are linear operators. In terms of bilinear forms, we can
write~\eqref{eq:two-fold} as
\begin{subequations}\label{eq:abstract:forms}
  \begin{align}
    a_1(\bu, \bw) + b_1(\bw, \bp) &= \inner{\bff}{\bw}  \label{eq:abstract:forms1} , \\
    a_2(\bv, \bz) + b_2(\bz, \bp) &= \inner{\br}{\bz} \label{eq:abstract:forms2} , \\
    b_1(\bu, \bq) + b_2(\bv, \bq) - a_3(\bp, \bq) &= \inner{\bg}{\bq} \label{eq:abstract:forms3} .
  \end{align}
\end{subequations}
This abstract form was studied in the context of twofold saddle point
problems and equivalence of inf-sup stability conditions by Howell and
Walkington~\cite{howell2011inf} for the case where $A_3 = A_2 = 0$.
%Note that Howell and Walkington~\cite{howell2011inf} considered the
%structure with $- B_1^T$ and $- B_2^T$, but that is equivalent by
%$\bm p \mapsto - \bm p$).

%%%%%%%%%%%%%%%%%%%%%%%%%%%%%%%%%%%%%%%%%%%%%%%%%%%%%%%%%%%%%%%%%%%%%%%%%%%%%%%%%%%%5a
\iffalse

By taking $\bm w = \bm u$, $\bm q = \bm p$ and $\bm v = \bm z$
in~\eqref{eq:abstract:forms}, adding the two first equations and
subtracting the third, and using Cauchy-Schwartz on the right-hand
sides, we obtain
\begin{equation}
  a_1(\bm u, \bm u) + a_2(\bm v, \bm v) + a_3(\bm p, \bm p) \leq
  \| f \| \| \bm u \| + \|\bm r \| \| \bm v \| + \| \bm g \| \| \bm p \|.
\end{equation}

\fi
%%%%%%%%%%%%%%%%%%%%%%%%%%%%%%%%%%%%%%%%%%%%%%%%%%%%%%%%%%%%%%%%%%%%%%%%%%%%%%%%%%%%5a

\subsection{Three-field variational formulation of the Biot-Brinkman system}

We consider the following variational formulation of the Biot-Brinkman
system~\eqref{eq:bb:matrix} with the boundary conditions given
by~\eqref{eq:bcs}: given $\boldsymbol f, \boldsymbol r, \boldsymbol
g$, find $(\bu, \bv, \bp) \in \bm U \times \bm V \times \bm P$ such
that~\eqref{eq:abstract:forms} holds with
\begin{subequations}\label{eq:forms}
  \begin{align}
    a_1(\bm u, \bm w) &= \inner{\sigma(\bm u)}{\varepsilon(\bm w)}, \\
    a_2(\bm v, \bm z) &= \ssum_{i=1}^n \inner{\gamma_i \varepsilon(\bm v_i)}{\varepsilon(\bm z_i)} + \inner{\tau K_i^{-1} \bm v_i}{\bm z_i} 
  %  \equiv \inner{\gamma \varepsilon(\bm v)}{\varepsilon(\bm z)} + \inner{R^{-1} \bm v}{\bm z} 
  \\
    a_3(\bm p, \bm q) &= 
    \ssum_{i=1}^n \inner{s_i p_i}{q_i}
    - \sum_{i, j=1}^n\inner{\tau \beta_{ij} p_{j}}{q_i} \\
    b_1(\bm w, \bm p) &= \ssum_{i=1}^n \inner{\div \bm w}{\alpha_i p_i} \equiv \inner{\div \bm w}{\bm \alpha \cdot \bm p} , \\
    b_2(\bm v, \bm q) &= \ssum_{i=1}^n \inner{\tau \div \bm v_i}{\bm q_i}, 
  \end{align}
\end{subequations}
for all $\bm w \in \bm U$, $\bm z \in \bm V$, and $\bm q \in \bm P$. Equivalently,
$(\bu, \bv, \bp) \in \bm U \times \bm V \times \bm P$ solves
\begin{equation}
  \label{eq:bb}
  \mathcal{A}((\bu, \bv, \bp), (\bw, \bz, \bq)) = \inner{(\boldsymbol f, \boldsymbol r, \boldsymbol g)}{(\bz, \bw, \bq)}
\end{equation}
for all $(\bz, \bw, \bq) \in \bm U \times \bm V \times \bm P$ where
\begin{equation}\label{eq:generalized_BB}
  \begin{split}
  \mathcal{A}((\bu, \bv, \bp), (\bw, \bz, \bq))
  = \, &a_1(\bu, \bw) + a_2(\bv, \bz) + b_1(\bw, \bp) + b_1(\bu, \bq) \\ &+ b_2(\bz, \bp) + b_2(\bw, \bq) - a_3(\bp, \bq) .
  \end{split}
\end{equation}
We refer to~\eqref{eq:abstract:forms}--\eqref{eq:forms}, or also \eqref{eq:bb}, as a three-field formulation of the
Biot-Brinkman system, with three-field referring to the three groups
of fields (displacement, fluxes and pressures).

%%%%%%%%%%%%%%%%%%%%%%%%%%%%%%%%%%%%%%%%%%%%%%%%%%%%%%%%%%%%%%%%%%%%%%%%%%%%%%%%%%%%%5a
%%%%%%%%%%%%%%%%%%%%%%%%%%%%%%%%%%%%%%%%%%%%%%%%%%%%%%%%%%%%%%%%%%%%%%%%%%%%%%%%%%%%%%5a

\subsection{Stability properties}

%\mer{Is there a particular reason for why we use semicolons to separate the arguments to e.g. $\mathcal{A}$ here? I suggest we just use usual commas.}

In this section we prove the main theoretical result of this paper, that is, the uniform well-posedness of problem~\eqref{eq:abstract:forms}--\eqref{eq:forms}
under the norms induced by~\eqref{norms}, as stated in \cref{continu_stability}. The proof utilizes the abstract framework for the stability analysis of
perturbed saddle-point problems that has recently been presented in~\cite{hong2021framework}. It is performed in two steps. In the first step, we recast the system
\eqref{eq:abstract:forms}--\eqref{eq:forms} into the following two-by-two (single) perturbed saddle-point problem
\begin{align}\label{eq:abstract_perturbedSPP}
 \mathcal A((\bu, \bv, \bm p),(\bw, \bz, \bm q)) & = 
 \mathcal A((\bar{\bu}, \bm p),(\bar{\bw},\bm q)) \\ &=a(\bar{\bu},\bar{\bw})
+b(\bar{\bw},\bm p) + b(\bar{\bu},\bm q) - c(\bm p,\bm q),\nonumber
 \end{align} 
where $\bar{\bu}=(\bu,\bv)$, $\bar{\bw}=(\bw,\bz)$ and 
\begin{align*}
a(\bar{\bu},\bar{\bw})&=a_1(\bu,\bw)+a_2(\bv,\bz),\\ 
b(\bar{\bw},\bm p)&=b_1(\bw,\bm p)+b_2(\bz,\bm p),\\ 
c(\bm p,\bm q)&=a_3(\bm p,\bm q),
\end{align*} 
with $a_1(\cdot,\cdot)$, $a_2(\cdot,\cdot)$, $a_3(\cdot,\cdot)$, $b_1(\cdot,\cdot)$ 
and $b_2(\cdot,\cdot)$ as defined in \eqref{eq:forms}.
Then, according to Theorem~5 in~\cite{hong2021framework}, for properly chosen seminorms
$\vert \cdot \vert_{\bm Q}$ and $\vert \cdot \vert_{\bar{\bm V}}$, which are specified in 
\cref{continu_stability_in_combined_norm} below, the uniform well-posedness of this problem is guaranteed
under the fitted (full) norms
\begin{align}
\Vert \bm q\Vert^2_{\bm Q} = \vert \bm q\vert_{\bm Q}^2+ c(\bm q,\bm q)= :
\langle \bar{Q}\bm q,\bm q \rangle_{\bm Q'\times \bm Q}, \label{q:full} \\
\Vert \bar{\bw}\Vert^2_{\bar{\bm V}} = \vert \bar{\bw}\vert_{\bar{\bm V}}^2 
+ \langle B \bar{\bw},\bar{Q}^{-1}B\bar{\bw} \rangle_{\bm Q'\times \bm Q},  \label{v:full}
\end{align}  
if the following two conditions are satisfied for positive constants $c_a$ and $c_b$ which are
independent of all model parameters:
\begin{equation}\label{a:coerc}
a(\bar{\bv},\bar{\bv}) \ge c_a \vert \bar{\bv} \vert_{\bar{\bm V}}^2, \qquad 
\forall \bar{\bv} \in \bar{\bm V},
\end{equation}
 \begin{equation}\label{b:inf_sup}
\sup_{\bar{\bv}\in \bar{\bm V}}\frac{b(\bar{\bv},\bm q)}{\Vert \bar{\bv} \Vert_{\bar{\bm V}}} \ge 
c_b \vert \bm q \vert_{\bm Q}, \qquad \forall \bm q\in \bm Q.
\end{equation} 
This means that under the conditions~\eqref{a:coerc} and \eqref{b:inf_sup} %for the bilinearforms~$a(\cdot,\cdot)$ and $b(\cdot,\cdot)$
the bilinear form in~\eqref{eq:abstract_perturbedSPP} satisfies the estimates
\begin{equation}\label{continuity_in_comb_norm}
|\mathcal A((\bu, \bv,\bm p),(\bw,\bz,\bm q))|\le C_b \| (\bu,\bv,\bm p) \|_{\bar{\bm X}} \| (\bw,\bz,\bm q) \|_{\bar{\bm X}},
\end{equation}
and
\begin{align}\label{stability_in_comb_norm} 
\inf_{({\boldsymbol u},{\boldsymbol v},{\boldsymbol p})\in {\boldsymbol X}}
\sup_{({\boldsymbol w},{\boldsymbol z},{\boldsymbol q})\in {\boldsymbol X}}
\frac{\mathcal{A}(({\boldsymbol u},{\boldsymbol v},{\boldsymbol p}),({\boldsymbol w},{\boldsymbol z},{\boldsymbol q}))}
{ \| (\bu,\bv,\bm p) \|_{\bar{\bm X}}\| (\bw,\bz,\bm q) \|_{\bar{\bm X}}} \geq \omega,
\end{align}
for the combined norm $\| (\cdot,\cdot,\cdot) \|_{\bar{\bm X}}$ defined by
\begin{equation}\label{eq:combined_norm}
\Vert (\bw,\bz,\bm q) \Vert_{\bar{\bm X}}^2 :=  \Vert \bm q\Vert_{\bm Q} ^2 + \Vert \bar{\bw}\Vert_{\bar{\bm V}}^2
\end{equation}
on the space $\bm X = {\boldsymbol U} \times {\boldsymbol V} \times {\boldsymbol P}$
with constants $C_b$ and $\omega$ that do not depend on any of the model parameters.

Before we turn to the proof of
estimates~\eqref{continuity_in_comb_norm} and
\eqref{stability_in_comb_norm} in
Theorem~\ref{continu_stability_in_combined_norm} below, we recall appropriate 
inf-sup conditions for the spaces ${\boldsymbol U}$, ${\boldsymbol
  V}$, ${\boldsymbol P}$ in Lemma~\ref{divinf-sup}.
\begin{lemma}\label{divinf-sup}
The following conditions hold with constants $\beta_d > 0$ and $\beta_s > 0$:
  \begin{eqnarray}
\inf_{q\in P_i} \sup_{{\boldsymbol v} \in {\boldsymbol V}_i} \frac{({\rm div} {\boldsymbol v}, q)}{\|{\boldsymbol v}\|_1\|q\|}
\geq \beta_d, \quad i=1,\dots,n, \label{eq:inf_sup_d} \\
\inf_{(q_1,\cdots,q_n)\in P_1\times\cdots\times P_n}
\sup_{{\boldsymbol u}\in {\boldsymbol U}}
\frac{\left({\rm div} {\boldsymbol u}, \sum\limits_{i=1}^n q_i\right)}{\|{\boldsymbol u}\|_1\left\|\sum\limits_{i=1}^n q_i\right\|} \geq \beta_s  .
\label{eq:inf_sup_s}
\end{eqnarray}
\end{lemma}
\begin{proof}
  See~\cite{Brezzi1974existence,Boffi2013mixed}.
\end{proof}

\begin{theorem}\label{continu_stability_in_combined_norm}
Consider problem~\eqref{eq:abstract:forms}--\eqref{eq:forms} on the space
$\bm X = {\boldsymbol U} \times {\boldsymbol V} \times {\boldsymbol P}= \bar{\boldsymbol V} \times {\boldsymbol Q}$ and define
the combined norm $\Vert \cdot \Vert_{\bar{\bm X}}$ via \eqref{eq:combined_norm} where the fitted norms $\Vert \cdot \Vert_{\bm Q}$ and
$\Vert \cdot \Vert_{\bar{\bm V}}$ are defined by
\eqref{q:full}--\eqref{v:full} with seminorms 
\begin{align}
\vert \bm q\vert^2_{\bm Q} 
& = ((\Lambda_3+\Lambda_4)\bm q,\bm q), \label{q:semi}\\
\vert \bar{\bw}\vert^2_{\bar{\bm V}} & = a(\bar{\bw},\bar{\bw}). \label{v:semi}
\end{align}
Then, the continuity and stability estimates \eqref{continuity_in_comb_norm} and \eqref{stability_in_comb_norm} hold with positive 
constants $C_b$ and $\omega$ that are independent of all model parameters.
\end{theorem}

\begin{proof}
To prove statement~\eqref{continuity_in_comb_norm}, one uses the Cauchy-Schwarz inequality and the definition of the norms. 

In order to prove \eqref{stability_in_comb_norm} we verify the conditions of Theorem 5 in~\cite{hong2021framework}, i.e., conditions~\eqref{a:coerc} and~\eqref{b:inf_sup}. Noting that $\vert \bar{\bw}\vert^2_{\bar{\bm V}} = a(\bar{\bw},\bar{\bw})$, we find that condition~\eqref{a:coerc} trivially holds with $c_a=1$ so it remains to show~\eqref{b:inf_sup}. The bilinear form $b$ is induced by the operator $B:\bar{\bm V}\rightarrow \bm Q'$ that is given by
\begin{align*}
B & = \begin{pmatrix}
-\alpha_1\div & -\tau\div & 0 & 0& \ldots & 0 \\
-\alpha_2\div & 0 & -\tau\div & 0 & \ldots & 0 \\
-\alpha_3\div & 0 & 0 & -\tau\div & \ldots  & 0  \\
\vdots & \vdots & \vdots & \vdots & \ddots &  \vdots  \\
-\alpha_n\div & 0 & 0 &0& \ldots  & -\tau\div
\end{pmatrix}.
\end{align*}
Thanks to \cref{divinf-sup}, for a given $(\bar{\bu},\bm p)$ we can
choose test functions $\bar{\bw}=(\bw,\bz)$ such that
\begin{align*}
-\div \bw = \frac{1}{2\mu+\lambda}\sum_{i=1}^n \alpha_ip_i, \quad 
\Vert \bw \Vert_1 \le \beta_s^{-1} \frac{1}{2\mu+\lambda} \Vert 
\sum_{i=1}^n \alpha_ip_i\Vert ,
\end{align*} 
\begin{align*}
-\div \bz_i = \tau R p_i, \quad 
\Vert \bz_i \Vert_1 \le \beta_s^{-1}\tau R \Vert p_i\Vert, \quad i=1,\ldots,n.
\end{align*} 

With these choices we find that
\begin{align*}
b(\bar{\bw},\bm p) & = - (\div \bw, \sum_{i=1}^n\alpha_i p_i)- \sum_{i=1}^n  (\tau \div 
\bz_i, p_i) \\
& = \frac{1}{2\mu+\lambda} \left(\sum_{i=1}^n\alpha_i p_i,\sum_{i=1}^n \alpha_ip_i \right)
+ \sum_{i=1}^n (\tau^2 Rp_i,p_i) \\
& = (\Lambda_4 \bm p,\bm p)+(\Lambda_3 \bm p,\bm p)  = \vert \bm p\vert_{\bm Q}^2.
\end{align*}

In view of~\eqref{v:full} and noting that $\langle B \bar{\bw},\bar{Q}^{-1}B\bar{\bw} \rangle_{\bm Q'\times \bm Q}
= (\Lambda^{-1}B \bar{\bw},\bar{\bw})$, we obtain 
\begin{align*}
\Vert \bar{\bw}\Vert_{\bar{\bm V}}^2 & = 2 \mu(\varepsilon(\bw),\varepsilon(\bw))
+ \lambda (\div \bw,\div \bw)
+\sum_{i=1}^n \gamma_i (\varepsilon(\bz_i),\varepsilon(\bz_i)) \\
& \quad + \sum_{i=1}^n (\tau K_i^{-1}\bz_i,\bz_i)+(\Lambda^{-1}B \bar{\bw},B \bar{\bw})\\
& \le \beta_s^{-2} (2\mu+\lambda)\left(\frac{1}{2\mu+\lambda}\right)^2 
\Vert \sum_{i=1}^n \alpha_ip_i\Vert^2 + \sum_{i=1}^n \gamma_i \beta_s^{-2} \tau^2 R^{2} 
\Vert p_i\Vert^2 \\
& \quad + \sum_{i=1}^n \tau K_i^{-1}\beta_s^{-2} \tau^2 R^{2}\Vert p_i\Vert^2 +(\Lambda^{-1}B \bar{\bw},B \bar{\bw}) \\
& \le \beta_s^{-2} \frac{1}{2\mu+\lambda} \Vert \sum_{i=1}^n \alpha_ip_i\Vert^2 
+ \beta_s^{-2} \sum_{i=1}^n (\gamma_i+\tau K_i^{-1})\tau^2 R^{2} \Vert p_i\Vert^2  +(\Lambda^{-1}B \bar{\bw},B \bar{\bw}) \\
& \le \beta_s^{-2} \frac{1}{2\mu+\lambda} \Vert \sum_{i=1}^n \alpha_i p_i\Vert^2 + \beta_s^{-2} 
\sum_{i=1}^n \tau^2 R \Vert p_i\Vert^2 +(\Lambda^{-1}B \bar{\bw},B \bar{\bw}) \\
& \le  \beta_s^{-2} \left((\Lambda_4 \bm p,\bm p)+(\Lambda_3 \bm p,\bm p)\right)
+((\Lambda_3+\Lambda_4)^{-1}B \bar{\bw},B \bar{\bw}) \\
& \le (\beta_s^{-2}+1) \vert \bm p\vert_{\bm Q}^2,
\end{align*} 
where we have also used $(\Lambda^{-1}B \bar{\bw},B \bar{\bw}) \le ((\Lambda_3+\Lambda_4)^{-1}B \bar{\bw},B \bar{\bw})$
and $B \bar{\bw}=(\Lambda_3+\Lambda_4)\bm p$. Finally,~\eqref{b:inf_sup} follows from 
\begin{align*}
\sup_{\bar{\bv}\in \bar{\bm V}}\frac{b(\bar{\bv},\bm q)}{\Vert \bar{\bv} \Vert_{\bar{\bm V}}} \ge 
\frac{b(\bar{\bw},\bm q)}{\Vert \bar{\bw} \Vert_{\bar{\bm V}}} \ge 
\frac{1}{\sqrt{\beta_s^{-2}+1}}\frac{\vert \bm q \vert_{\bm Q}^2}{\vert \bm q \vert_{\bm Q}}= c_b \vert \bm q \vert_{\bm Q},\quad \forall \bm q\in \bm Q.
\end{align*} 
\end{proof}

We have now established the well-posedness of the Biot-Brinkman problem under the specific combined norm $\Vert \cdot \Vert_{\bar{\bm X}}$ of the
form \eqref{eq:combined_norm}, specified through~\eqref{q:semi} and~\eqref{v:semi}. Next, we show that this combined norm is equivalent to the norm $\Vert \cdot \Vert_{\bm X}$ defined by
\begin{align}\label{equiv_norm}
\Vert  ({\boldsymbol w},{\boldsymbol z},{\boldsymbol q}) \Vert^2_{\bm X} :=
\Vert {\boldsymbol w} \Vert_{\boldsymbol U}^2 + \Vert {\boldsymbol z} \Vert_{\boldsymbol V}^2 + \Vert {\boldsymbol q} \Vert_{\boldsymbol P}^2.
\end{align}
The following Lemma is useful in establishing this norm equivalence, cf.~\cite[Lemma 2.1]{hong2020parameter} where the statement has been
proven for $\boldsymbol\alpha=(1,1,\dots,1)^T$.
\begin{lemma}\label{etauu:7}
For any $a>0$ and $b>0$ and $\boldsymbol \alpha=(\alpha_1,\dots,\alpha_n)^T$, we have that
\begin{equation}
(a I_{n\times n}+b \boldsymbol \alpha\boldsymbol \alpha^T)^{-1}=a^{-1}I-a^{-1}(ab^{-1}+\bm \alpha^T\bm \alpha)^{-1}\bm \alpha\bm \alpha^T ,
\end{equation}
and 
\begin{equation}
\bm \alpha^T(a I_{n\times n}+b \boldsymbol \alpha\boldsymbol \alpha^T)^{-1}\bm \alpha=\frac{\bm \alpha^T\bm \alpha}{ab^{-1}+\bm \alpha^T\bm \alpha}b^{-1}\le b^{-1}.
\end{equation}
\end{lemma}
\begin{proof}
The proof follows the lines of the proof of~Lemma 2.1 in~\cite{hong2020parameter}.
\end{proof}
Now we can establish the following norm equivalence result.
\begin{lemma}\label{norm_equivalence}
The norm~\eqref{equiv_norm} defined in terms of~\eqref{norms} is equivalent to the combined norm~\eqref{eq:combined_norm}
based on \eqref{q:semi} and \eqref{v:semi}.
%~Theorem~\ref{continu_stability_in_combined_norm}.
\end{lemma}
\begin{proof}
First, we note that  
\begin{align*}
B \bar{\bw}&=
\begin{pmatrix}
-\alpha_1{\rm div} \bw-\tau\div \bm z_1\\
-\alpha_2{\rm div} \bw-\tau\div \bm z_2\\
\vdots\\
-\alpha_n{\rm div} \bw-\tau\div \bm z_n
\end{pmatrix}
=-{\rm div} \bw
\begin{pmatrix}
\alpha_1\\
\alpha_2\\
\vdots\\
\alpha_n
\end{pmatrix}
+\tau
\begin{pmatrix}
-\div \bm z_1\\
-\div \bm z_2\\
\vdots\\
-\div \bm z_n
\end{pmatrix} \\
&\equiv -\bm \alpha{\rm div} \bw -\tau {\rm Div} \bm z.
 \end{align*} 
Then for any $1> \epsilon > 0$, by Cauchy's inequality, we obtain 
\begin{align*}
 &(\Lambda^{-1}B \bar{\bw},B \bar{\bw})\\
 &= \left(\Lambda^{-1} (\bm \alpha{\rm div} \bw+\tau {\rm Div} \bm z),(\bm \alpha{\rm div} \bw+\tau {\rm Div} \bm z)\right)\\
 &= (\Lambda^{-1} \bm \alpha{\rm div} \bw, \bm \alpha{\rm div} \bw)+ 2(\Lambda^{-1} \bm \alpha{\rm div} \bw, \tau {\rm Div} \bm z)+ (\Lambda^{-1} \tau {\rm Div} \bm z, \tau {\rm Div} \bm z)\\
% & \ge (\Lambda^{-1} -\bm \alpha{\rm div} \bw, -\bm \alpha{\rm div} \bw)-\epsilon^{-1}  (\Lambda^{-1} -\bm \alpha{\rm div} \bw, -\bm \alpha{\rm div} \bw)-\epsilon (\Lambda^{-1} \tau {\rm Div} \bm z, \tau {\rm Div} \bm z)+(\Lambda^{-1} \tau {\rm Div} \bm z, \tau {\rm Div} \bm z)\\
 &\ge -(\epsilon^{-1}-1) (\Lambda^{-1} \bm \alpha{\rm div} \bw, \bm \alpha{\rm div} \bw)+(1-\epsilon)(\Lambda^{-1} \tau {\rm Div} \bm z, \tau {\rm Div} \bm z)\\
 &\ge -(\epsilon^{-1}-1) ((\Lambda_3+\Lambda_4)^{-1} \bm \alpha{\rm div} \bw, \bm \alpha{\rm div} \bw)+(1-\epsilon)(\Lambda^{-1} \tau {\rm Div} \bm z, \tau {\rm Div} \bm z).
 \end{align*} 
By \cref{etauu:7}, with $a=\tau^2R, b=\frac{1}{2\mu+\lambda}$, we have
\begin{align*}
 &(\Lambda^{-1}B \bar{\bw},B \bar{\bw})\\
 &\ge -(\epsilon^{-1}-1)  ((\Lambda_3+\Lambda_4)^{-1} \bm \alpha{\rm div} \bw, \bm \alpha{\rm div} \bw)+(1-\epsilon)(\Lambda^{-1} \tau {\rm Div} \bm z, \tau {\rm Div} \bm z)\\
 &=-(\epsilon^{-1}-1)  (\bm \alpha^T(\Lambda_3+\Lambda_4)^{-1} \bm \alpha{\rm div} \bw, {\rm div} \bw)+(1-\epsilon)(\Lambda^{-1} \tau {\rm Div} \bm z, \tau {\rm Div} \bm z)\\
 &\ge -(\epsilon^{-1}-1) (2\mu +\lambda) (\div \bm w, \div \bm w)+(1-\epsilon)(\Lambda^{-1} \tau {\rm Div} \bm z, \tau {\rm Div} \bm z).
 \end{align*} 
 Therefore, we get
 \begin{align*}
\Vert \bar{\bw}\Vert_{\bar{\bm V}}^2 & = 2\mu(\varepsilon(\bw),\varepsilon(\bw))
+ \lambda (\div \bw,\div \bw)x
+\sum_{i=1}^n \gamma_i (\varepsilon(\bz_i),\varepsilon(\bz_i))\\
& \quad + \sum_{i=1}^n (\tau K_i^{-1}\bz_i,\bz_i)+(\Lambda^{-1}B \bar{\bw},B \bar{\bw})\\
&\ge 2\mu (\varepsilon(\bw),\varepsilon(\bw))
+ \lambda (\div \bw,\div \bw) -(\epsilon^{-1}-1) (2\mu +\lambda) (\div \bm w, \div \bm w)\\
&\quad +\sum_{i=1}^n \gamma_i (\varepsilon(\bz_i),\varepsilon(\bz_i))
 + \sum_{i=1}^n (\tau K_i^{-1}\bz_i,\bz_i) +(1-\epsilon)(\Lambda^{-1} \tau {\rm Div} \bm z, \tau  {\rm Div} \bm z).
\end{align*} 
Now, for $\epsilon=\frac23$, we obtain 
\begin{align*}
\Vert \bar{\bw}\Vert_{\bar{\bm V}}^2
&\ge 2\mu (\varepsilon(\bw),\varepsilon(\bw))
+ \lambda (\div \bw,\div \bw) -\frac12 (2\mu +\lambda) (\div \bm w, \div \bm w)\\
&\quad +\sum_{i=1}^n \gamma_i (\varepsilon(\bz_i),\varepsilon(\bz_i))
 + \sum_{i=1}^n (\tau K_i^{-1}\bz_i,\bz_i) +\frac13 (\Lambda^{-1} \tau^2{\rm Div} \bm z, {\rm Div} \bm z)\\
 &\ge \frac12\left( 2\mu (\varepsilon(\bw),\varepsilon(\bw))
+ \lambda (\div \bw,\div \bw) \right)\\
&\quad +\frac13 \left(\sum_{i=1}^n \gamma_i (\varepsilon(\bz_i),\varepsilon(\bz_i))
 + \sum_{i=1}^n (\tau K_i^{-1}\bz_i,\bz_i) + (\Lambda^{-1}\tau^2 {\rm Div} \bm z, {\rm Div} \bm z)\right),
\end{align*}
namely 
$
\|\bm w\|^2_{\boldsymbol  U}+ \|\bm z\|^2_{\boldsymbol  V}\lesssim  \Vert \bar{\bw}\Vert_{\bar{\bm V}}^2.
$
On the other hand, it is obvious that 
$$
\Vert \bar{\bw}\Vert_{\bar{\bm V}}^2 \lesssim \|\bm w\|^2_{\boldsymbol  U}+ \|\bm z\|^2_{\boldsymbol  V}.
$$
Together, this gives
$
\Vert \bar{\bw}\Vert_{\bar{\bm V}}^2 \cong \|\bm w\|^2_{\boldsymbol  U}+ \|\bm z\|^2_{\boldsymbol  V}.
$
\end{proof}
In view of \cref{continu_stability_in_combined_norm} and \cref{norm_equivalence}, we conclude that the Biot-Brinkman problem
is also well-posed under the norm~\eqref{equiv_norm} defined in terms of~\eqref{norms}. We summarize our results in the following theorem.
\begin{theorem}\label{continu_stability}
~
\begin{itemize}
\item[(i)] 
There exists a positive constant $C_{b}$ independent of  the parameters $\lambda$, $K_i^{-1}$, $s_i$, ${\beta}_{ij}$, 
$i,j \in \{1,\dots,n\}$, the network scale~$n$ and the time step $\tau$ such that the inequality 
\begin{equation*}
|\mathcal A((\bu,\bv,\bm p),(\bw,\bz,\bm q))|\le C_b (\|\bm u|_{\bm U}+\|\bm v\|_{\bm V}
+\|\bm p\|_{\bm P})  (\|\bm w\|_{\bm U}+\|\bm z\|_{\bm V}+\|\bm q\|_{\bm P})
\end{equation*}
holds true for any $(\bm u, \bm v,\bm p)\in \boldsymbol U\times \boldsymbol V\times \boldsymbol P , (\bm w, \bm z, \bm q)\in \boldsymbol U\times \boldsymbol V\times \boldsymbol P$.
\item[(ii)]  There is a constant $\omega> 0$ independent of the 
parameters $\lambda,K_i^{-1}, s_i, \beta_{ij}$, $i,j \in \{1,\dots,n\}$, the number of networks $n$ and the time step $\tau$ such that 
\begin{align*}%\label{stability} 
\inf_{({\boldsymbol u},{\boldsymbol v},{\boldsymbol p})\in {\boldsymbol X}}
\sup_{({\boldsymbol w},{\boldsymbol z},{\boldsymbol q})\in {\boldsymbol X}}
\frac{\mathcal{A}(({\boldsymbol u},{\boldsymbol v},{\boldsymbol p}),({\boldsymbol w},{\boldsymbol z},{\boldsymbol q}))}
{( \|{\boldsymbol u}\|_{{\boldsymbol U}}
+\|{\boldsymbol v}\|_{{\boldsymbol V}}+ \|{\boldsymbol p}\|_{{\boldsymbol P}})( \|{\boldsymbol w}\|_{{\boldsymbol U}}
+ \|{\boldsymbol z}\|_{{\boldsymbol V}}+ \|{\boldsymbol q}\|_{{\boldsymbol P}})} \geq \omega,
\end{align*}
where ${\boldsymbol X}:={\boldsymbol U} \times {\boldsymbol V}\times {\boldsymbol P}$.
\item[(iii)]  The MPET system~\eqref{eq:bb} has a unique solution $(\boldsymbol u, \boldsymbol v,\bp)\in \boldsymbol U\times \boldsymbol V\times \boldsymbol P$ and the following stability estimate holds:
\begin{equation*}%\label{MPET_stability_estimate}
\|\boldsymbol u\|_{\boldsymbol U}+\|\boldsymbol  v\|_{\boldsymbol V}+\|\boldsymbol p\|_{\boldsymbol P}\leq C_1 (\|\boldsymbol f\|_{\boldsymbol U'}
+\|\bm g\|_{\bm P'}),
\end{equation*} 
where $C_1$ is a positive constant independent of the parameters $\lambda,K_i^{-1}$, $s_i$, $\beta_{ij}, i,j \in \{1,\dots,n\}$, the network scale $n$ 
and the time step $\tau$, 
and
$\|\boldsymbol f\|_{\boldsymbol U'}=
\sup\limits_{\boldsymbol w\in \boldsymbol U}\frac{(\boldsymbol f, \boldsymbol w)}{\|\boldsymbol w\|_{\boldsymbol U}}$, $\|\boldsymbol g\|_{\boldsymbol P'}=
\sup\limits_{\boldsymbol q\in \boldsymbol P}\frac{(\boldsymbol g,\boldsymbol q)}
{\|\boldsymbol q\|_{\boldsymbol P}}=\|\Lambda^{-\frac{1}{2}} \boldsymbol g\|.
$
\end{itemize} 
\end{theorem}

%%%%%%%%%%%%%%%%%%%%%%%%%%%%%%%%%%%%%%%%%%%%%%%%

\section{Discrete generalized Biot-Brinkman problems}
\label{sec:uni_stab_disc_model} 

Stable and parameter-robust discretizations for the multiple network poroelasticity equations have been proposed based on a classical three-field formulation using a discontinuous Galerkin (DG) \cite{arnold2002unified,hong2019unified} formulation of the momentum equation resulting in strong mass conservation, see~\cite{hong2019conservative}, or based on a total pressure formulation in the setting of conforming methods in~\cite{lee2019mixed}. These discrete models have been developed as generalizations of the corresponding Biot models, see~\cite{HongKraus2017parameter} in case of conservative discretizations and \cite{Lee2016parameter} in case of the total pressure scheme. A hybridized version of the method in~\cite{HongKraus2017parameter} has recently been presented in~\cite{kraus2021uniformly}. For other conforming parameter-robust discretizations of the Biot model see also~\cite{chen2020robust,rodrigo2018new} and \cite{Kumar2020conservative}, where the latter method is based on a total pressure formulation introducing the flux as a fourth field, which then also results in mass conservation. In this paper we extend the approach from~\cite{hong2019conservative,HongKraus2017parameter} to obtain mass-conservative discretizations for the generalized Biot-Brinkman system~\eqref{eq:abstract:forms}--\eqref{eq:forms}, which generalizes the MPET system.

\subsection{Notation}

Consider a shape-regular triangulation $\mathcal{T}_h$ of the domain $\Omega$ into triangles/tetrahedrons, where the subscript $h$ indicates the mesh-size. 
Following the standard notation, we first denote the set of all interior edges/faces and the set of all boundary edges/faces of $\mathcal{T}_h$  %are denoted 
by $\mathcal{E}_h^{I}$ and $\mathcal{E}_h^{B}$ respectively, their union by $\mathcal{E}_h$ and then we define the broken Sobolev spaces  
$$
H^s(\mathcal{T}_h)=\{\phi\in L^2(\Omega), \mbox{ such that } \phi|_T\in H^s(T) \mbox{ for all } T\in \mathcal{T}_h \}
$$
for $s\geq 1$. 

Next we introduce the notion of jumps $[\cdot]$ and averages $\{\cdot \}$ as follows. 
For any $q\in H^1(\mathcal{T}_h)$, $\bm v \in H^1(\mathcal{T}_h)^d$ and $\bm \tau \in H^1(\mathcal{T}_h)^{d\times d}$ and any $e\in \mathcal{E}_h^{I}$ the jumps are given as
\begin{equation*}
[q]=q|_{\partial T_1\cap e}-q|_{\partial T_2\cap e},\quad
[\bm v]=\bm v|_{\partial T_1\cap e}-\bm v|_{\partial T_2\cap e}
\end{equation*}
and the averages as
\begin{equation*}
\begin{split}
\{\bm v\} &=\frac{1}{2}(\bm v|_{\partial T_1\cap e}\cdot \bm n_1-\bm
v|_{\partial T_2\cap e}\cdot \bm n_2), \quad 
\{\bm \tau\}=\frac{1}{2}(\bm \tau|_{\partial T_1\cap
e} \bm n_1-\bm \tau|_{\partial T_2\cap e} \bm n_2),
\end{split}
\end{equation*}
while for  
$e \in  \mathcal{E}_h^{B}$, 
\[[q]=q|_{e}, ~~ [\bm v]=\bm v|_{e},\quad
\{\bm v\}=\bm v |_{e}\cdot \bm n,\quad
\{\bm \tau\}=\bm \tau|_{e}\bm n.
\]
Here $T_1$ and $T_2$ are any two elements from the triangulation that share an edge or face $e$ while $\bm n_1$ and $\bm n_2$ 
denote the corresponding unit normal vectors to $e$ 
pointing to the exterior of $T_1$ and $T_2$, respectively. 

\subsection{Mixed finite element spaces and discrete formulation}\label{subsec:DGdiscretization}
We consider the following finite element spaces to approximate the displacement, fluxes and pressures:
\begin{eqnarray*}
\bm U_h&=&\{\bm u \in H(\div ,\Omega):\bm u|_T \in \bm U(T),~T \in \mathcal{T}_h;~ \bm u \cdot
\bm n=0~\hbox{on}~\partial \Omega\},
\\[1ex]
\bm V_{i,h}&=&\{\bm v \in H(\div ,\Omega):\bm v|_T \in \bm V_i(T),~T \in \mathcal{T}_h;~ \bm v \cdot
\bm n=0~\hbox{on}~\partial \Omega\},~~i=1,\dots,n,
\\
P_{i,h}&=&\left\{p \in L^2(\Omega):p|_T \in P_i(T),~T \in \mathcal{T}_h; ~\int_{\Omega}p dx=0\right\},~~i=1,\dots,n,
\end{eqnarray*}
%A comment here about what is done with the tangential part?
where $\bm U(T)/\bm V_{i}(T)/P_{i}(T)={\rm BDM}_{l}(T)/{\rm BDM}_{l}(T)/{\rm P}_{l-1}(T)$ for $l\ge 1$.
Note that for each of these choices $\div  \bm U(T)=\div  \bm V_i(T)=P_i(T)$ is fulfilled.
We remark that the tangential part of the displacement
  boundary condition \eqref{eq:bcs} is enforced by a Nitsche method,
  see e.g. \cite{hansbo2003discontinuous}. Furthermore the orthogonality
  constraint for the pressures in $P_{i,h}$ is realized in the implementation
  by introducing (scalar) Lagrange multipliers.

Let us denote
$
\bm v_h^T=(\boldsymbol v^T_{1,h}, \ldots, \boldsymbol v^T_{n,h}),\, \bm p_h^T=(p_{1,h},\ldots, p_{n,h}),\,%$$ 
%$$
\bm z_h^T=(\boldsymbol z^T_{1,h}, \ldots, \boldsymbol z^T_{n,h})$, $\bm q_h^T=(q_{1,h},\ldots, q_{n,h})$ and
$$
\bm V_h=\boldsymbol V_{1,h}\times\ldots\times\boldsymbol V_{n,h},
\quad \bm P_h= P_{1,h}\times\ldots\times P_{n,h}, \quad  \bm X_h = \bm U_h\times\bm V_h\times \bm P_h.
$$
The discretization of the variational problem~\eqref{eq:abstract:forms}--\eqref{eq:forms} now is given as follows: find $(\bu_h, \bv_h, \bp_h)\in  \bm X_h$, such that for any $(\bw_h, \bz_{h}, \bq_{h})\in  \bm X_{h}$ and $i=1,\dots, n$
\begin{subequations}
  \label{eq:75-77}
  \begin{align}
  a_h(\bm u_h,\bm w_h) +\lambda ( \div \boldsymbol  u_h, \div \boldsymbol  w_h)
  + (\boldsymbol \alpha \cdot \bp_{h}, \div \boldsymbol  w_h) &= (\boldsymbol f, \boldsymbol w_h),\label{eq:75}\\
  \gamma_i a_h(\bm v_{i,h},\bm z_{i,h}) + (\tau K^{-1}_i\bv_{i,h},\bz_{i,h}) + (p_{i,h},\tau \divv \bz_{i,h}) &= 0,  \label{eq:76} \\
  (\divv\bu_h,\alpha_i q_{i,h}) + (\tau \divv\bv_{i,h},q_{i.h})  - s_i (p_{i,h},q_{i,h}) & \nonumber \\
  + \sum_{j=1}^{n}\tau\beta_{ij} (p_{j,h}, q_{i,h}) &= (g_i,q_{i,h}), \label{eq:77} 
\end{align}
\end{subequations}
where
\begin{equation}
  \begin{split}
  a_h(\bm \phi,\bm \psi)=
  \sum _{T \in \mathcal{T}_h} \int_T \ep(\bm{\phi}) :
  \ep(\bm{\psi}) \dx -\sum_{e \in \mathcal{E}_h} \int_e \{\ep(\bm{\phi})\} \cdot [\bm \psi_t] \ds\\
  - \sum _{e \in \mathcal{E}_h} \int_e \{\ep(\bm{\psi})\} \cdot [\bm \phi_t] \ds+\sum _{e
    \in \mathcal{E}_h} \int_e \eta h_e^{-1}[ \bm \phi_t] \cdot [\bm \psi_t] \ds,
  \end{split}
  \label{78}
\end{equation}
and $\eta $ is a stabilization parameter independent of all other problem
parameters, the network scale $n$ and the mesh size $h$. 

We note that the discrete variational problem~\eqref{eq:75-77} has been derived for the weak formulation~\eqref{eq:bb} with
homogeneous boundary conditions. For general rescaled boundary conditions with DG discretizations we refer the reader 
to e.g.~\cite{HongKraus2017parameter}.

\subsection{Stability properties}
\label{subsec:discrete_stability}
For any function $\bm \phi\in \bm H^2(\mathcal{T}_h):=H^2(\mathcal{T}_h)^d$, consider the following mesh dependent norms 
%Let $\bm u$ be a function from $\boldsymbol U_h$ and consider 
%the mesh dependent norms

%\mer{Is there a reason for the subscript $0$ in the norms below? Other places we just use no subscript for $L^2$.}

\begin{eqnarray*}
\|\bm{\phi}\|_h^2&=&\sum _{T \in \mathcal{T}_h} 
\|\ep(\bm{\phi})\|_{T}^2+\sum _{e \in \mathcal{E}_h} h_e^{-1}\|[ \bm \phi_t]\|_{e}^2, \\
\|\bm \phi\|_{1,h}^2&=&\sum _{T \in \mathcal{T}_h} \|\nabla\bm{\phi}\|_{T}^2+\sum _{e \in \mathcal{E}_h} h_e^{-1}\|[ \bm{\phi}_t]\|_{e}^2,
\end{eqnarray*}
and
\begin{equation}\label{DGnorm}
\|\bm \phi\|^2_{DG}=\sum _{T \in \mathcal{T}_h} \|\nabla\bm{\phi}\|_{T}^2+\sum _{e \in \mathcal{E}_h} h_e^{-1}\|[ \bm{\phi}_t]\|_{e}^2+\sum _{T \in \mathcal{T}_h}h_T^2|\bm{\phi}|^2_{2,T}.
\end{equation}
Details about the well-posedness and approximation properties of the DG formulation of elasticity, Stokes and Brinkman-type systems 
can be found in~\cite{hong2016robust, honguniformly}.  

Now, for $\bm u\in  H({\rm div}, \Omega) \cap \boldsymbol H^2(\mathcal T_h)$, we define the norm
\begin{equation}\label{U_hnorm}
\|\bm u\|^2_{\bm U_h}=\|\bm u\|^2_{DG}+\lambda \|\div \bm u\|^2
\end{equation}
and  for $\bm v\in H({\rm div}, \Omega) \cap \boldsymbol H^2(\mathcal T_h)$, we define the norm 
\begin{equation}\label{V_hnorm}
\|{\boldsymbol v}\|^2_{\boldsymbol V_h}= \sum_{i=1}^n\big(\gamma_i \| \boldsymbol v_i\|^2_{DG}+(\tau K_i^{-1}{\boldsymbol v}_i,{\boldsymbol v}_i)\big)
+ (\Lambda^{-1}  {\text{Div\,}} {\boldsymbol v} ,{\text{Div\,}} {\boldsymbol v}).
\end{equation}
The well-posedness and approximation properties of the DG formulation are detailed in~\cite{hong2016robust, honguniformly}.  
Here we briefly present some important results:
\begin{itemize}
\item $\|\cdot\|_{DG}$, $\|\cdot\|_h$, and $\|\cdot\|_{1,h}$ are equivalent on $\bm U_h$; that is
$$
\|\bm{u}_h\|_{DG}\eqsim  \|\bm{u}_h\|_h\eqsim\|\bm u_h\|_{1,h},\,\mbox{for all}~\bm u_h \in \bm U_h.
$$
\item 
$a_h$ from~\eqref{78} is continuous and it holds true that
\begin{eqnarray}\label{continuity:a_h}
|a_h(\bm u,\bm w)|&\lesssim& \| \bm u  \|_{DG}  \| \bm w  \|_{DG},\quad\mbox{for all}\quad \bm u,~\bm w\in \bm H^2(\mathcal{T}_h).
\end{eqnarray}
\item The following inf-sup conditions are satisfied
\begin{equation}\label{inf-sup}
\begin{split}
 & \inf_{(q_{1,h},\cdots,q_{n,h})\in P_{1,h} \times \cdots \times P_{n,h}}\sup_{\bm{u}_h\in \bm{U}_h}\frac{(\operatorname{div}\bm{u}_h,\sum\limits_{i=1}^n 
 q_{i,h})}{\|\bm{u}_h\|_{1,h}\|\sum\limits_{i=1}^n q_{i,h}\|}\geq \beta_{sd},\\
 &  \inf_{q_{i,h}\in P_{i,h}}\sup_{\bm{v}_{i,h}\in \bm{V}_{i,h}}\frac{(\operatorname{div}\bm{v}_{i,h},q_{i,h})}{\|\bm{v}_{i,h}\|_{1,h}\|q_{i,h}\|}\geq \beta_{dd}, \quad i=1,\dots,n.
\end{split}
\end{equation}

\end{itemize}
Using the definition of the matrices $\C_1$ and $\C_2$, next we define the bilinear form 
\begin{equation}\label{eq:79}
\begin{split}
&\mathcal A_h((\bu_h, \bv_{h}, \bm p_{h}),(\bw_h, \bz_{h}, \bm q_{h}))=a_h(\bm u_h,\bm w_h)+\lambda ( \div \boldsymbol  u_h, \div \boldsymbol  w_h) 
\\& +\sum_{i=1}^{n} (\alpha_i p_{i,h},{\rm div} \bw_h) +\sum_{i=1}^{n}\gamma_i a_h(\bm v_{i,h},\bm z_{i,h})\\ & +\sum_{i=1}^{n}(\tau K^{-1}_i\bv_{i,h},\bz_{i,h})
+ \tau (\bp_h, \Divv \bz_h)
\\&
+\sum_{i=1}^{n}(\divv\bu_h, \alpha_i q_{i,h})+\tau (\Divv \bv_h,\bq_h) - ((\C_1+\C_2)\bp_h,\bq_h)
\end{split}
\end{equation}
related to problem~\eqref{eq:75}--\eqref{eq:77}.

We equip $\bm X_h$ with the norm defined by $\|(\cdot,\cdot,\cdot)\|_{\bm X_h}^2:=\|\bm \cdot \|_{\bm U_h}^2+\|\bm \cdot \|_{\bm V_h}^2
+\|\bm \cdot \|_{\bm P}^2$. Similar to \cref{continu_stability}, 
the following uniform stability result holds:
%the following uniform stability results can be found as in~\cite{Hong2018conservativeMPET}:
\begin{theorem}\label{boundd}
~
\begin{itemize}
\item[(i)] For any $\bm u_h, \bm w_h\in \bm U_h; \, \bm v_h, \bm z_h\in \bm V_h; \, \bm p_h, \bm q_h\in \bm P_h$
there exists a positive constant $C_{bd}$ independent of all model parameters, 
the network scale~$n$ and the mesh size~$h$ such that the inequality 
\begin{equation*}
|\mathcal A_h((\bu_h,\bv_h,\bm p_h),(\bw_h,\bz_h,\bm q_h))|\le C_{bd} \|(\bm u_h,\bm v_h, \bm p_h)\|_{\bm X_h}  
\|(\bm w_h, \bm z_h, \bm q_h)\|_{\bm X_h}
\end{equation*}
holds true.
\item[(ii)]  There exists a constant $\omega_d>0$ independent of all discretization and model parameters 
such that 
\begin{equation}%\label{eq:80}
\displaystyle\inf_{(\boldsymbol u_h,  \boldsymbol v_h, \bm p_h)\in  \bm X_h} 
\sup_{(\boldsymbol w_h,\boldsymbol z_h, \bm q_h)\in \boldsymbol X_h}\frac{\mathcal A_h((\bu_h,\bv_{h},\bm p_{h}),(\bw_h,\bz_{h},\bm q_{h}))}
{\|(\bm u_h,\bm v_h, \bm p_h)\|_{\bm X_h}  
\|(\bm w_h, \bm z_h, \bm q_h)\|_{\bm X_h}
}\geq \omega_d.
\end{equation}

\item[(iii)] Let $(\boldsymbol u_h,  \boldsymbol v_h, \bm p_h)\in \bm X_h$ solve~\eqref{eq:75}-\eqref{eq:77}
and
$$\|\boldsymbol f\|_{\boldsymbol U_h'}=
\sup\limits_{\boldsymbol w_h\in \boldsymbol U_h}\frac{(\boldsymbol f, \boldsymbol w_h)}
{\|\boldsymbol w_h\|_{\boldsymbol U_h}},\quad \|\bm g\|_{\bm P'}=
\sup\limits_{\bm q_h\in \bm P_h}\frac{(\bm g, \bm q_h)}{\|\bm q_h\|_{\bm P}}.$$
Then the estimate
\begin{equation*}
\|\boldsymbol u_h\|_{\boldsymbol U_h}+\|\boldsymbol  v_h\|_{\boldsymbol V}+\|\bm p_h\|_{\bm P}\leq C_2 (\|\boldsymbol f\|_{\boldsymbol U_h'}+\|\bm g\|_{\bm P'})
\end{equation*} 
holds with 
a constant $C_2$ independent of the network scale $n$, the mesh size~$h$, the time step $\tau$ and the parameters 
$\lambda$, $K_i^{-1}$, $s_i$, ${\beta}_{ij}$, $i,j \in \{1,\dots,n\}$.
\end{itemize}
\end{theorem}

\subsection{Error estimates}\label{sec:error_estimates}

This subsection summarizes the error estimates that follow from the stability results presented in \Cref{subsec:discrete_stability}.
\begin{theorem}\label{error0}
Assume that $(\boldsymbol u,\boldsymbol v,\boldsymbol p)\in \bm U\cap \bm H^2(\mathcal T_h)\times \bm V\cap \bm H^2(\mathcal T_h)\times \bm P $ is the unique solution of~\eqref{eq:abstract:forms}--\eqref{eq:forms}, and let
$(\boldsymbol u_h,\boldsymbol v_h,\boldsymbol p_h)$ be the solution of \eqref{eq:75-77}. Then
the error estimates
\begin{equation} \label{eq:erroruv}
\|\boldsymbol u-\boldsymbol u_h\|_{\boldsymbol U_h}
+\|\boldsymbol v-\boldsymbol v_h\|_{\boldsymbol V_h}
\lesssim  \inf\limits_{\boldsymbol w_h\in \boldsymbol U_h, \boldsymbol z_h\in \boldsymbol V_h}
\Big(\|\boldsymbol u-\boldsymbol w_h\|_{\boldsymbol U_h}+\|\boldsymbol v-\boldsymbol z_h\|_{\boldsymbol V_h}\Big),
\end{equation}
%and 
\begin{equation} \label{eq:errorp}
\|\boldsymbol p-\boldsymbol p_h\|_{\boldsymbol P}\lesssim  \inf\limits_{\boldsymbol w_h\in \boldsymbol U_h, \boldsymbol z_h\in \boldsymbol V_h, \boldsymbol q_h\in {\boldsymbol P_h}}
\Big(\|\boldsymbol u-\boldsymbol w_h\|_{\boldsymbol U_h}+\|\boldsymbol v-\boldsymbol z_h\|_{\boldsymbol V_h}
+\|\boldsymbol p-\boldsymbol q_h\|_{\boldsymbol P}\Big),
\end{equation}
hold true, where the inequality constants are independent of the parameters $\lambda,K_i^{-1}, s_i$, ${\beta}_{ij}$ for $ i,j=1,\dots,n$,
the network scale~$n$, the mesh size~$h$ and the time step $\tau$. 
\end{theorem}
\begin{proof}
The proof of this result is analogous to the proof of Theorem 5.2 in~\cite{HongKraus2017parameter}.
\end{proof}

\begin{remark}
In particular, the above theorem shows that the proposed discretizations are locking-free.
%The error estimates provide near-best approximation results independent of any model parameters.
Note that estimate~\eqref{eq:erroruv} controls the error in $\bm u$
plus the error in $\bm v$ by the sum of the errors of the corresponding best approximations whereas estimate~\eqref{eq:errorp}
requires the best approximation errors of all three vector variables $\bm u$, $\bm v$ and $\bm p$ to control the error in $\bm p$.
\end{remark}

\subsection{A norm equivalent preconditioner}

We consider the following block-diagonal operator
\begin{equation}\label{Preconditioner:B}
\mathcal{B}:=\left[\begin{array}{ccc}
\mathcal{B}_{\bm u} & \bm 0 & \bm 0 \\
\bm 0 & \mathcal{B}_{\bm v}& \bm 0 \\
\bm 0& \bm 0& \mathcal{B}_{\bm p}  
\end{array}
\right]^{-1},
\end{equation}
where 
$$
\mathcal{B}_{\bm u}
=-\divv \ep-\lambda \nabla \divv,
$$
\begin{align*}
\mathcal{B}_{\bm v}=&  
\begin{bmatrix}
- \gamma_1{\text{div\,}} {\boldmath \epsilon} +\tau K_1^{-1}I&0  & \dots &0  \\
0 &- \gamma_2{\text{div\,}} {\boldmath \epsilon} +\tau K_2^{-1} I& \dots &0 \\
\vdots &\vdots  & \ddots & \vdots  \\
0& 0 & \dots &- \gamma_n{\text{div\,}} {\boldmath \epsilon} +\tau K_n^{-1}I
\end{bmatrix} \\ 
&
-
\begin{bmatrix}
\tilde{\Lambda}_{11}\nabla{\rm div} &\tilde{\Lambda}_{12}\nabla{\rm div} & \dots &\tilde{\Lambda}_{1n}\nabla{\rm div}\\
\tilde{\Lambda}_{21}\nabla{\rm div} &\tilde{\Lambda}_{22}\nabla{\rm div} & \dots &\tilde{\Lambda}_{2n}\nabla{\rm div}\\
\vdots &\vdots  & \ddots & \vdots  \\
\tilde{\Lambda}_{n1}\nabla{\rm div} &\tilde{\Lambda}_{n2}\nabla{\rm div} & \dots &\tilde{\Lambda}_{nn}\nabla{\rm div}
\end{bmatrix}
\end{align*}
and
\begin{align*}
\mathcal{B}_{\bm p}&=
\begin{bmatrix}
\Lambda_{11}I&\Lambda_{12}I& \dots &\Lambda_{1n}I\\
\Lambda_{21}I&\Lambda_{22}I& \dots &\Lambda_{2n}I\\
\vdots&\vdots  &\ddots&\vdots\\
\Lambda_{n1}I&\Lambda_{n2}I& \dots &\Lambda_{nn}I\\
\end{bmatrix}.
\end{align*}
Here, $\Lambda_{ij}$, $\tilde{\Lambda}_{ij}$, $i,j=1,\ldots,n$ are the entries of $\Lambda$ and $\Lambda^{-1}$, respectively.

As substantiated in~\cite{hong2019conservative}, the stability results
for the operator $\mathcal{A}$ in \eqref{eq:generalized_BB} imply that
the operator~$\mathcal{B}$ is a uniform norm-equivalent (canonical)
block-diagonal preconditioner that is robust with respect to all model
and discretization parameters. Note that $\mathcal{B}$ defines a
canonical uniform block-diagonal preconditioner on the continuous as
well as on the discrete level as long as discrete inf-sup conditions
analogous to~\eqref{eq:inf_sup_d} and~\eqref{eq:inf_sup_s} are
satisfied, cf.~\cite{hong2019conservative}.

%%%%%%%%%%%%%%%%%%%%%%%%%%%%%%%%%%%%%%%%%%%%%%%%%%%%%%%%%%%%%%%%%%%%%%%%%%%%%%%%%%%%%%%%%%%%%%%%%%%%%%%%%%%%%%%%%%%%%%%%%%%%%

\section{Numerical experiments}
\label{sec:numerics}
In this section we present numerical experiments whose results corroborate
stability properties of the finite element discretization of the generalized
Biot-Brinkman model (see \Cref{sec:error_estimates}) and the preconditioner
\eqref{Preconditioner:B}. We shall first demonstrate parameter robustness of the exact
preconditioner through a sensitivity study of the conditioning of the
preconditioned Biot-Brinkman system. Afterwards, scalable realization of
the preconditioner in terms multilevel methods for the displacement and
flux blocks is discussed. For simplicity, all the experiments concern the domain
$\Omega=(0, 1)^2$. The implementation was carried in the Firedrake finite element
framework \cite{firedrake}.

\subsection{Error estimates} We consider a single network, $n=1$, case
of the generalized Biot-Brinkman model \eqref{eq:generalized_BB},
with parameters $\mu=1$, $\tau=10^{-1}$, $\alpha_1=10^{-3}$ and $c_1=10^{-2}$
fixed (arbitrarily) while $K_1$, $\nu_1$ and $\lambda$ shall be varied in
order to test robustness of the error estimates established in \Cref{sec:error_estimates}.
To this end, we solve \eqref{eq:bb:matrix} with the right hand side
computed based on the exact solution
\begin{equation}\label{eq:mms}
\bu = \left(\frac{\partial \phi}{\partial y}, -\frac{\partial \phi}{\partial x}\right),\quad
\bv_1 = \nabla \phi_1,\quad
p_1 = \sin \pi(x-y),
\end{equation}
where
\[
    \phi = x^2 (x-1)^2 y^2 (y-1)^2, \quad
    \phi_1 = x^4 (x-1)^4 y^4 (y-1)^4.
\]
It can be seen that the manufactured solution satisfies the homogeneous conditions
$\bu|_{\partial\Omega}=\boldsymbol{0}$, $\bv_1\cdot\bn |_{\partial\Omega}=0$ for $\Omega=(0, 1)^2$.

Using discretization by $\text{BDM}_1$ elements for $\bU_h$, $\bV_{1, h}$ and piece-wise constant elements for the pressure space $P_{1, h}$, \cref{fig:error_varLambda}--\cref{fig:error_varK} show the errors of the numerical approximations in the parameter-dependent norms \eqref{U_hnorm}, \eqref{V_hnorm} and $\lVert\cdot \rVert_{\bP}$ defined in \eqref{norms-p} when one of the parameters $\lambda$, $K_1$ and $\nu_1$ is varied. In all the cases the expected linear convergence can be observed. In particular, the rate is independent of the parameter variations. We note that the error here is computed on a finer mesh than the finite element solution in order to prevent aliasing.

\begin{figure}
  \centering
  \includegraphics[height=0.1555\textheight]{{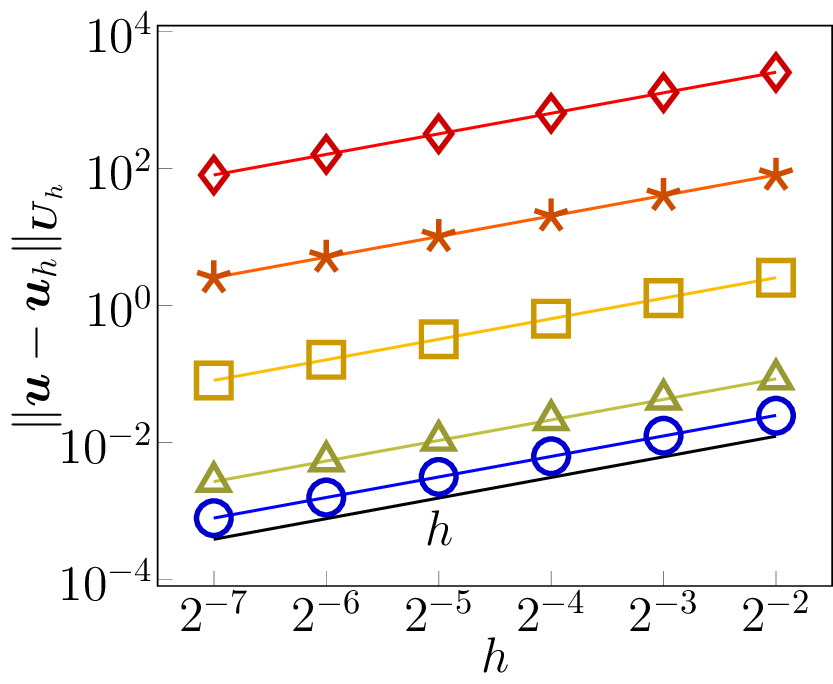}}
  \includegraphics[height=0.15\textheight]{{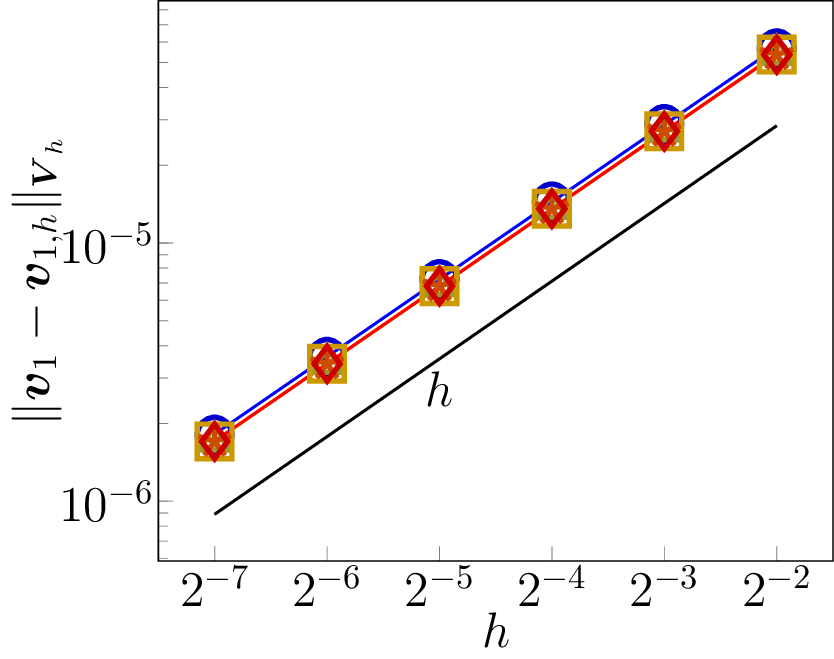}}
  \includegraphics[height=0.1555\textheight]{{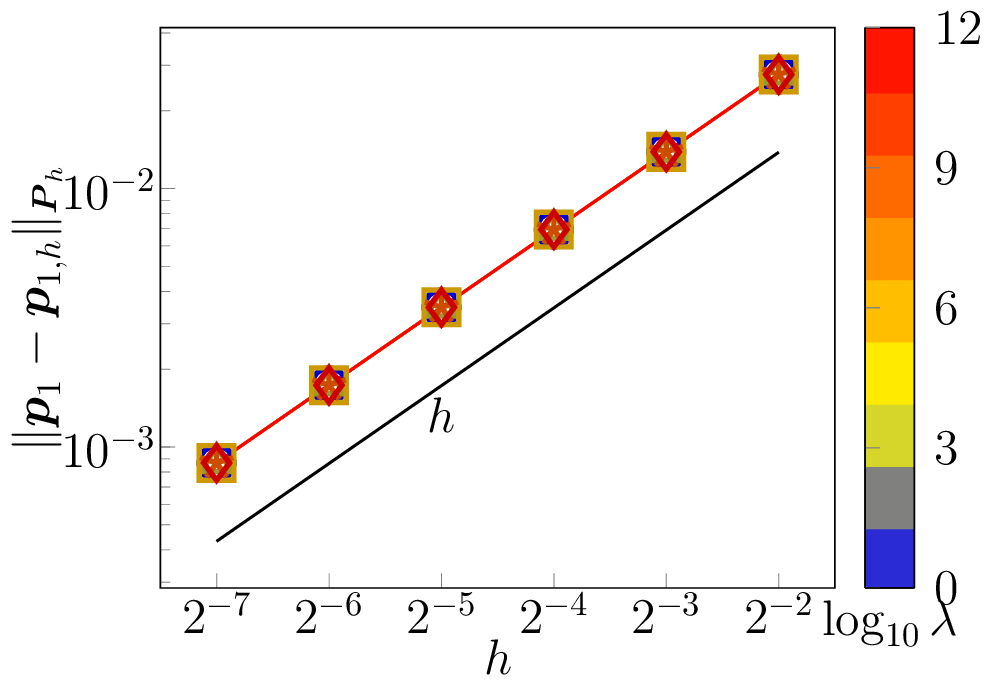}}
  \vspace{-10pt}
  \caption{
    Error approximation of the $\text{BDM}_1$-$\text{BDM}_1$-$\text{P}_0$ discretization
    of the single network Biot-Brinkman model. Parameters $\mu=1$, $\tau=10^{-1}$, $\alpha_1=10^{-3}$, $c_1=10^{-2}$,
    $\nu_1=1$ and $K_1=1$ are fixed. Line colors correspond to different values of $\lambda$. 
  }
  \label{fig:error_varLambda}
\end{figure}

\begin{figure}
  \centering
  \includegraphics[height=0.15\textheight]{{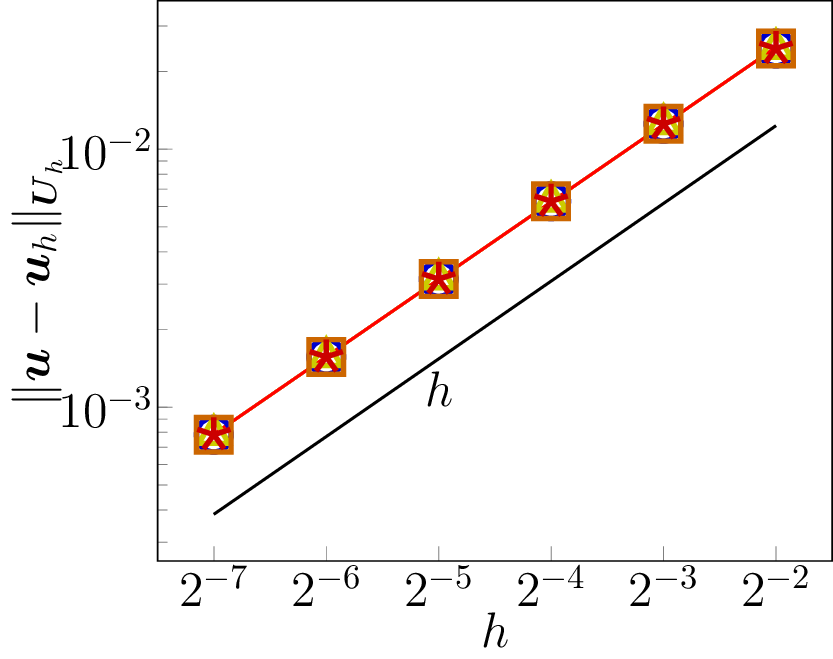}}
  \includegraphics[height=0.15\textheight]{{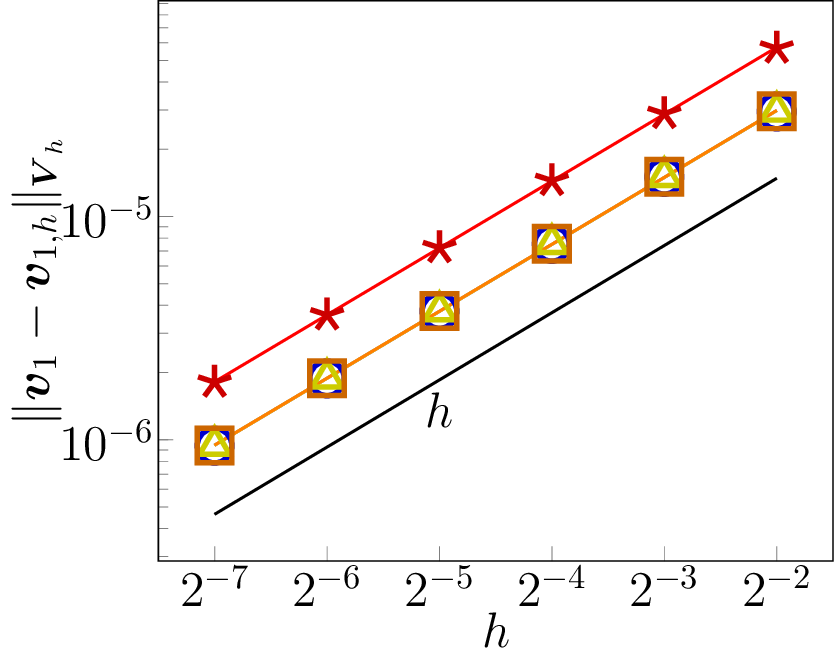}}
  \includegraphics[height=0.1555\textheight]{{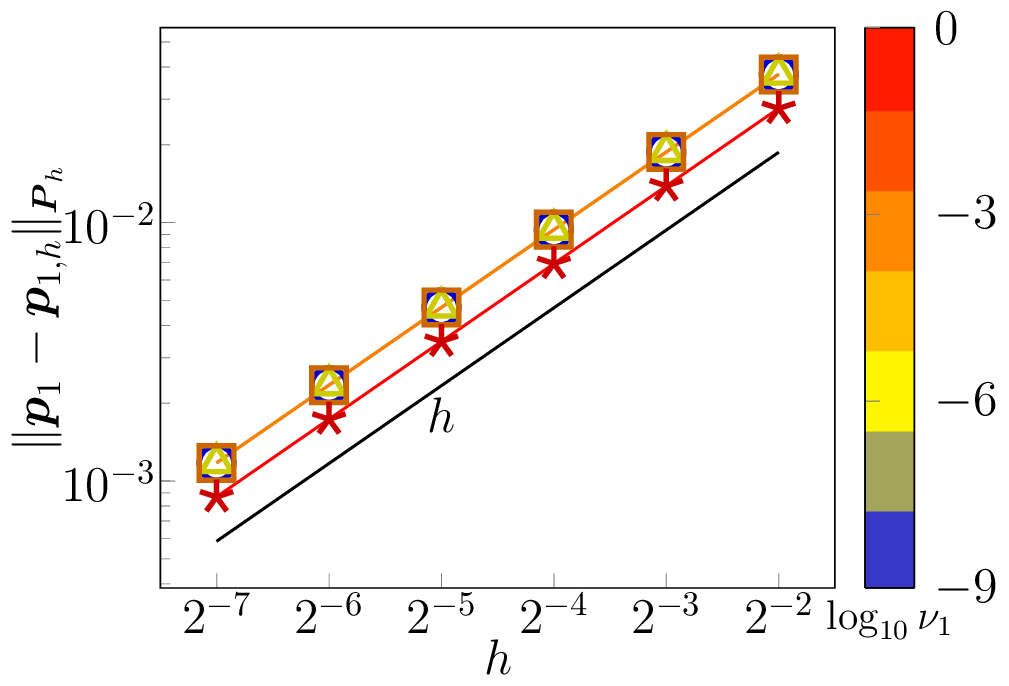}}
  \vspace{-20pt}
  \caption{
    Error approximation of the $\text{BDM}_1$-$\text{BDM}_1$-$\text{P}_0$ discretization
    of the single network Biot-Brinkman model. Parameters $\mu=1$, $\tau=10^{-1}$, $\alpha_1=10^{-3}$, $c_1=10^{-2}$,
	\kam{$K_1=1$} and $\lambda=1$ are fixed. Line colors correspond to different values of $\nu_1$. 
  }
  \label{fig:error_varNu}
\end{figure}

\begin{figure}
  \centering
  \includegraphics[height=0.15\textheight]{{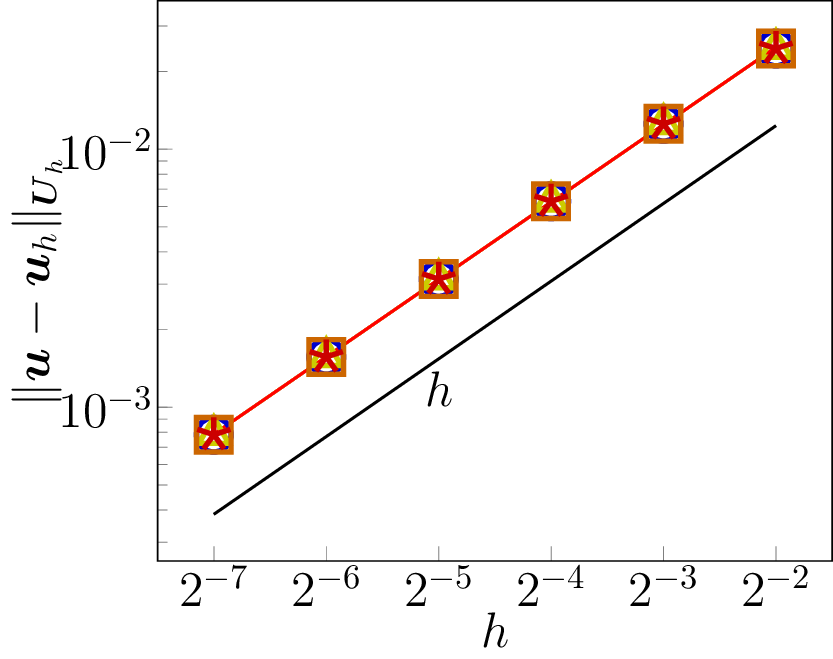}}
  \includegraphics[height=0.15\textheight]{{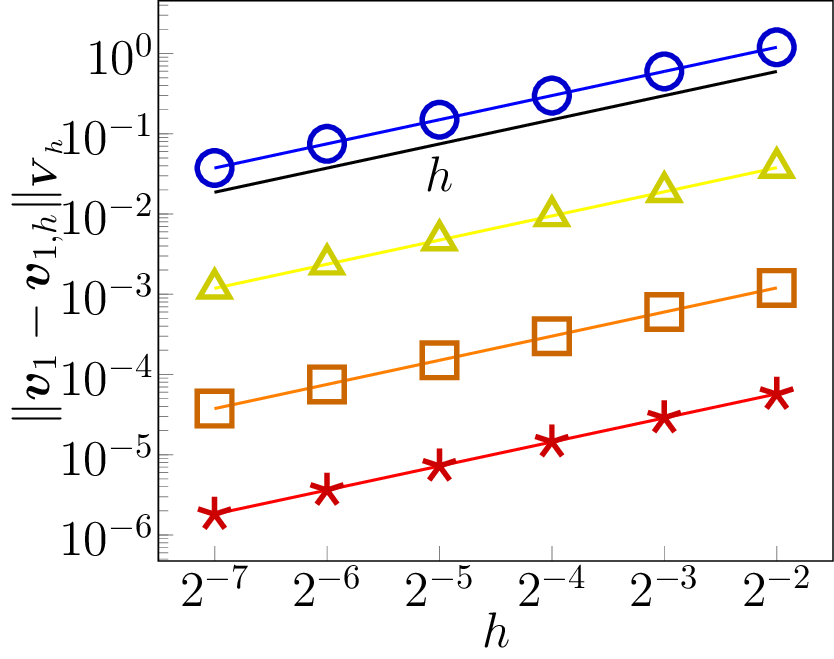}}
  \includegraphics[height=0.1555\textheight]{{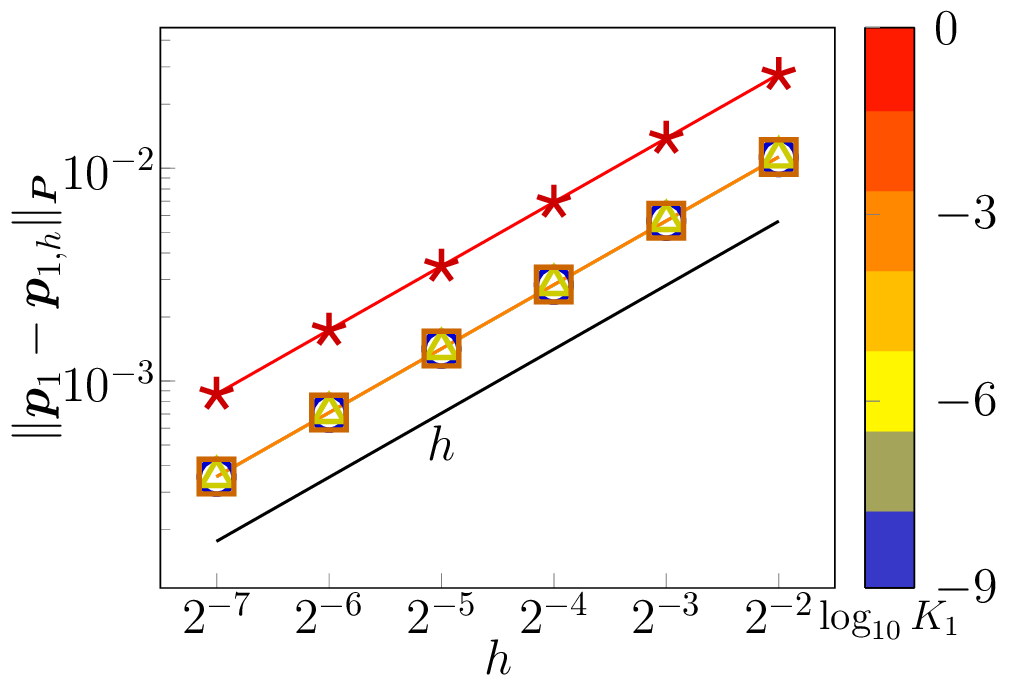}}
  \vspace{-20pt}  
  \caption{
    Error approximation of the $\text{BDM}_1$-$\text{BDM}_1$-$\text{P}_0$ discretization
    of the single network Biot-Brinkman model. Parameters $\mu=1$, $\tau=10^{-1}$, $\alpha_1=10^{-3}$, $c_1=10^{-2}$,
    $\nu_1=1$ and $\lambda=1$ are fixed. Line colors correspond to different values of $K_1$.     
  }
  \label{fig:error_varK}
\end{figure}

\subsection{Robustness of exact preconditioner}\label{sec:sensitivity} We verify robustness of the
canonical preconditioner \eqref{Preconditioner:B} using a generalized Biot-Brinkman system with two
networks. As the parameter space then counts 12 parameters in total we shall for simplicity fix material
properties of one of the networks (below we choose the network $i=1$)
to unity in addition to setting $\mu=1$, $\tau=1$. This choice leaves parameters
$\lambda$, $c_2$, $\alpha_2$, $\nu_2$, $K_2$ as well as the transfer coefficient
$\beta:=\beta_{12}$ to be varied. In the following experiments we let $1\leq\lambda\leq 10^{12}$,
$10^{-9}\leq \nu_2, K_2, \alpha_2 \leq 1$, $10^{-6}\leq \beta \leq 10^6$ and $c_2\in\left\{0, 1\right\}$
in order to perform a systematic sensitivity study. We note that we do not vary directly the scaling
parameters introduced in \eqref{eq:generalized_BB} but instead change the material parameters
in \eqref{eq:mBB:t}.

For the above choice of parameters the two-network problem is considered on the domain
$\Omega=(0, 1)^2$ with boundary conditions $\bu=\boldsymbol{0}$ on the left and right sides and
%traction 
$\left(\sig+\boldsymbol{\alpha}\cdot\boldsymbol{p}\boldsymbol{I}\right)\cdot\bn=\boldsymbol{0}$
%is prescribed
on the remaining part of the boundary; similarly, the Dirichlet conditions $\bv_i\cdot\bn=0$, $i=1, 2$
on the fluxes are prescribed only on the left and right sides.

Having constructed spaces $\bU_h$, $\bV_{1, h}$ $\bV_{2, h}$ with
$\text{BDM}_1$ elements and pressure spaces $P_{1, h}$ $P_{2, h}$ in
terms of piece-wise constants our results are summarized in
\cref{fig:a1E0}--\cref{fig:a1E-8} where slices of the explored
parameter space are shown. It can be seen that the condition numbers
remain bounded. Concretely, given discrete operators $\mathcal{A}_h$,
$\mathcal{B}_h$ that respectively discretize \eqref{eq:generalized_BB}
and the preconditioner \eqref{Preconditioner:B} the condition number
is computed based on the generalized eigenvalue problem $\mathcal{A}_h
x_k=\lambda_k \mathcal{B}^{-1}_h x_k$ as $\max_k \lvert \lambda_k
\rvert/\min_k \lvert \lambda_k \rvert.$ The higher condition numbers
(of about 8.5) are typically attained when $c_2=0$, $\lambda=1$ and
$\beta\ll 1$. We remark that with $c_2=0$ and all parameters but
$\beta$ set to 1 the condition number of $\Lambda$ ranges from $2.64$
when $\beta=10^{-6}$ to about $10^6$ when $\beta=10^{6}$. %We remark
%that the performance of the
%preconditioners did not detoriorate if the parameter ranges were altered, in particular,
%larger values of $\lambda$ and smaller values of $\nu_2$, $K_2$ and $\alpha_2$ did not present an issue.

\begin{figure}
  \centering
  \includegraphics[width=\textwidth]{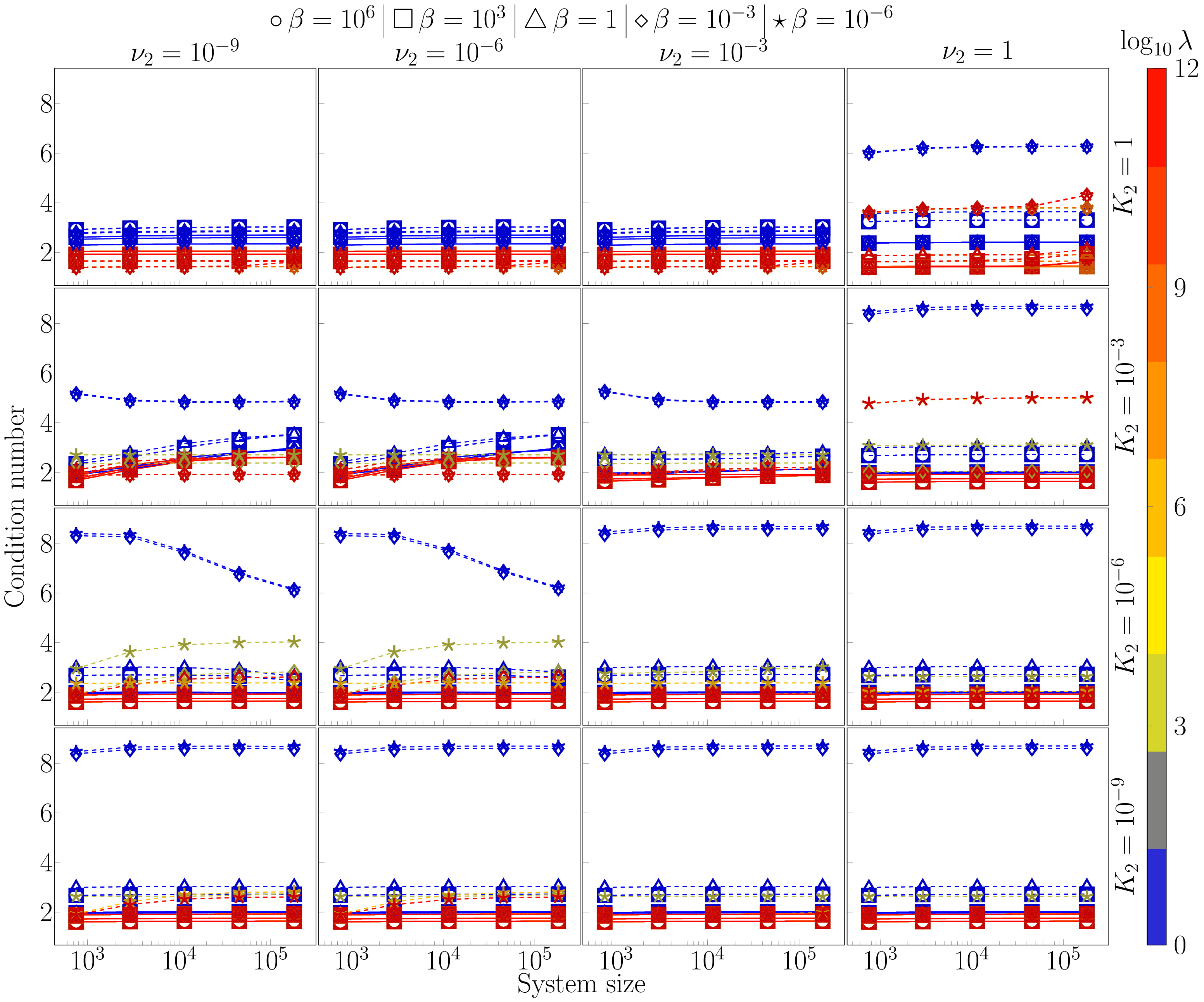}
  \vspace{-20pt}
  \caption{
    Performance of Biot-Brinkman preconditioner \eqref{Preconditioner:B} for $\alpha_2=1$ and
    varying parameters $\lambda$, $\nu_2$, $K_2$, $\beta$ (denoted by
    markers). Binary storage capacity is considered: $c_2=1$ (solid lines),
    $c_2=0$ (dashed lines). The remaining parameters are fixed at 1. Discretization
    by $\text{BDM}_1$-$(\text{BDM}_1)^2$-$(\text{P}_0)^2$ elements. Highest condition
    numbers correspond to $\beta\ll 1$ and $c_2=0$, $\lambda=1$.
  }
  \label{fig:a1E0}
\end{figure}

\begin{figure}
  \centering
  \includegraphics[width=\textwidth]{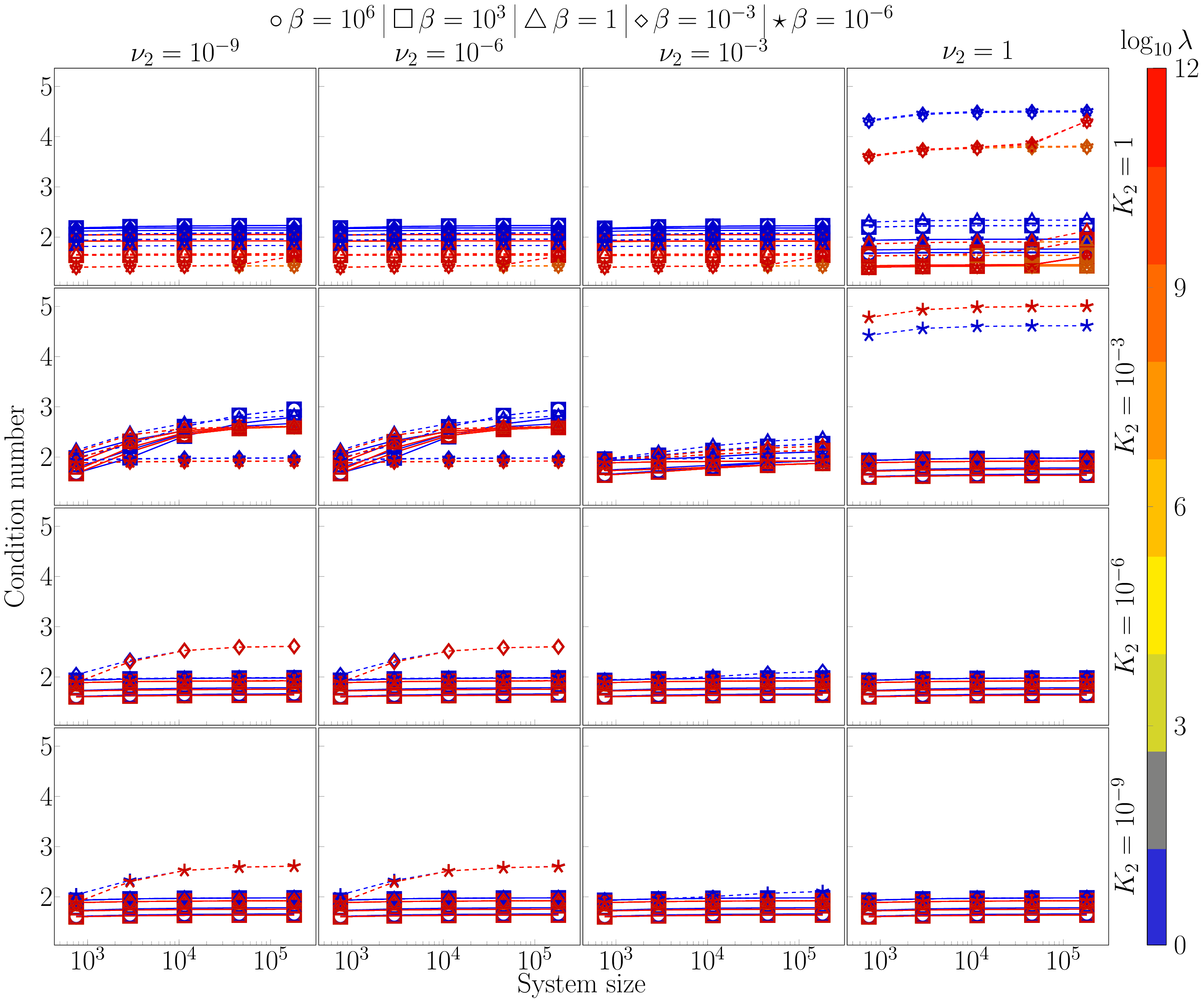}
  \vspace{-20pt}        
  \caption{
    Performance of Biot-Brinkman preconditioner \eqref{Preconditioner:B}
    for $\alpha_2=10^{-4}$ and varying parameters $\lambda$, $\nu_2$, $K_2$, $\beta$ (denoted by
    markers). Binary storage capacity is considered: $c_2=1$ (solid lines),
    $c_2=0$ (dashed lines). The remaining parameters are fixed at 1. Discretization
    by $\text{BDM}_1$-$(\text{BDM}_1)^2$-$(\text{P}_0)^2$ elements.
  }
  \label{fig:a1E-4}
\end{figure}

\begin{figure}
  \centering
  \includegraphics[width=\textwidth]{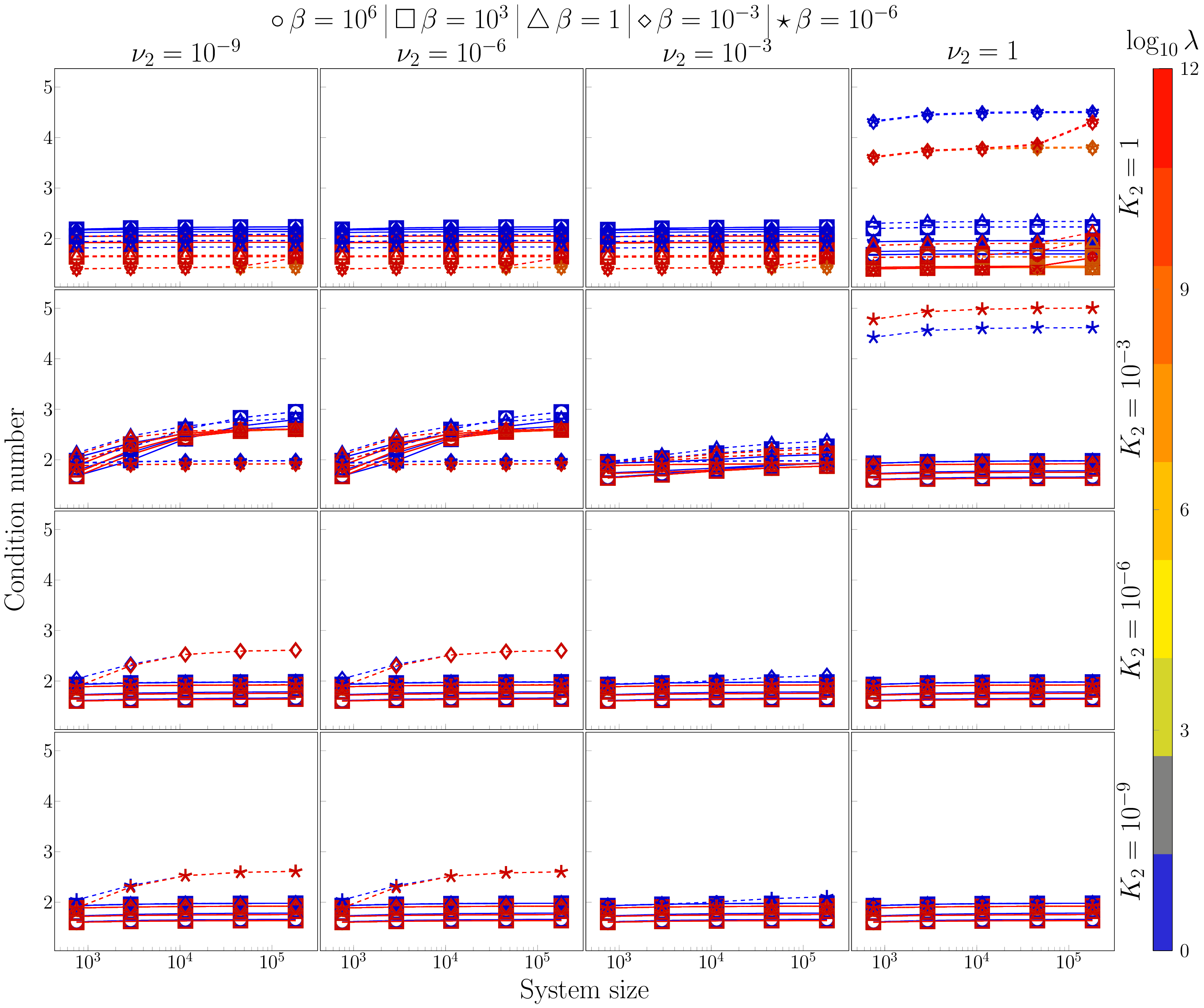}
  \vspace{-20pt}        
  \caption{
    Performance of Biot-Brinkman preconditioner \eqref{Preconditioner:B} for $\alpha_2=10^{-8}$ and
    varying parameters $\lambda$, $\nu_2$, $K_2$, $\beta$ (denoted by
    markers). Binary storage capacity is considered: $c_2=1$ (solid lines),
    $c_2=0$ (dashed lines). The remaining parameters are fixed at 1. Discretization
    by $\text{BDM}_1$-$(\text{BDM}_1)^2$-$(\text{P}_0)^2$ elements.
  }
  \label{fig:a1E-8}
\end{figure}

\subsection{Multigrid preconditioning}\label{sec:mg} Having seen that the exact
preconditioner \eqref{Preconditioner:B} yields parameter-robustness let us next
discuss possible construction of a scalable approximation of the operator $\mathcal{B}$.
Here, in order to approximate $\mathcal{B}_{\bm u}$ and $\mathcal{B}_{\bm v}$, 
we follow \cite{honguniformly, arnold2000multigrid, pcpatch} and employ
vertex-star relaxation schemes as part of geometric multigrid $F(2, 2)$-cycle
for the elastic block and $W(2, 2)$-cycle for the flux block. Numerical
experiments documenting robustness of the cycles for their respective blocks
are reported in Appendix \ref{sec:mg_components}.

% Example, the issue with bcs
To test performance of the multigrid-based preconditioner $\mathcal{B}$ 
we consider the two-network system from \Cref{sec:sensitivity} where
we set $c_2=0$, $\alpha_2=1$, $\beta\in\left\{10^{-6}, 10^{6}\right\}$ while
the remaining parameters are fixed to unity. We remark that for these parameter
values the highest condition numbers are attained with the exact preconditioner, cf.
\cref{fig:a1E0}. Furthermore, differing from the setup of the sensitivity
study, we (strongly) enforce $\bu\cdot\bn=0$ and $\bv_i\cdot\bn=0, i=1, 2$, on the entire boundary\footnote{
The reason for not prescribing the complete displacement vector as a boundary condition are limitations
in the PCPATCH framework which was used to implement the multigrid algorithm. In particular,
the software currently lacks support
for exterior facet integrals (see e.g. \cite{aznaran2021transformations})
which are required with BDM elements to weakly enforce conditions on the tangential
displacement by the Nitsche method.}. As before, the finite element discretization
is based on the $\text{BDM}_1$ and $\text{P}_0$ elements. 

% 3 levels, smoothers, PC patch, tests of components
In \cref{fig:flux_brinkman_beta1E6} and \cref{fig:flux_brinkman_beta1E-6} we report
the dependence on the mesh size and parameter values of the iteration counts of
the preconditioned MinRes solver where as the preconditioner both the exact Riesz map \eqref{Preconditioner:B}
and the multigrid-based approximation are used. More specifically, the multigrid cycles
for the displacement and flux blocks use 3 grid levels applying the exact $L^2$-projection
as the transfer operator. For both $\mathcal{B}_{\bm u}$ and $\mathcal{B}_{\bm v}$ the
vertex-star relaxation uses damped Richardson smoother. Comparing the results
we observe that the use of multigrid in \eqref{Preconditioner:B} translates to a
slight (about 1.5x) increase in the number of Krylov iterations compared to the
exact preconditioner. However, the iterations appear bounded in the mesh size 
and the parameter variations.

\begin{figure}
  \centering
  \includegraphics[width=\textwidth]{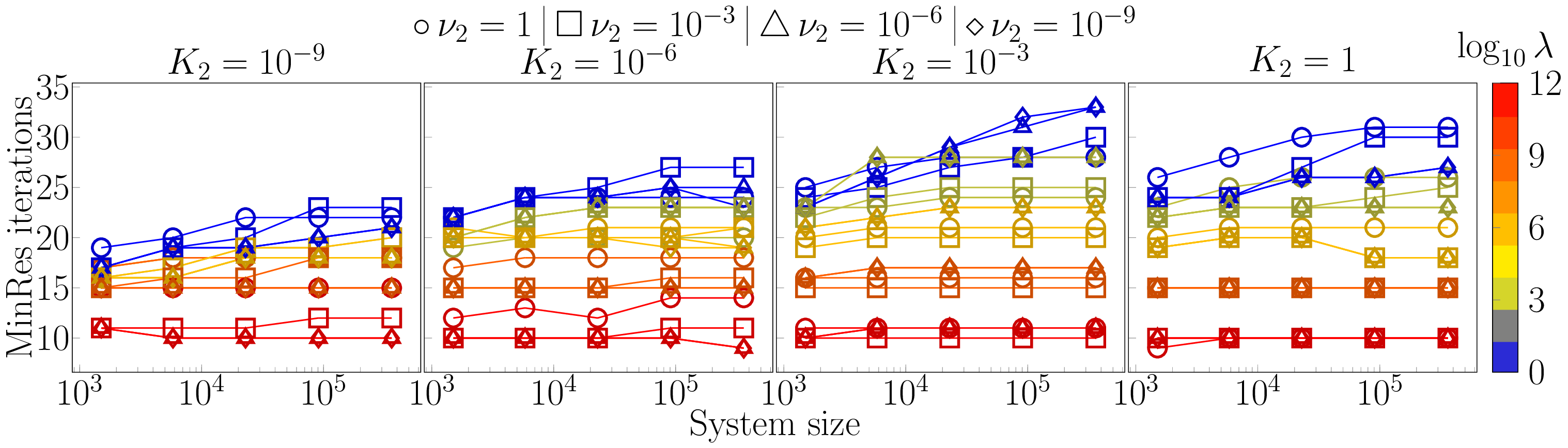}\\
  \includegraphics[width=\textwidth]{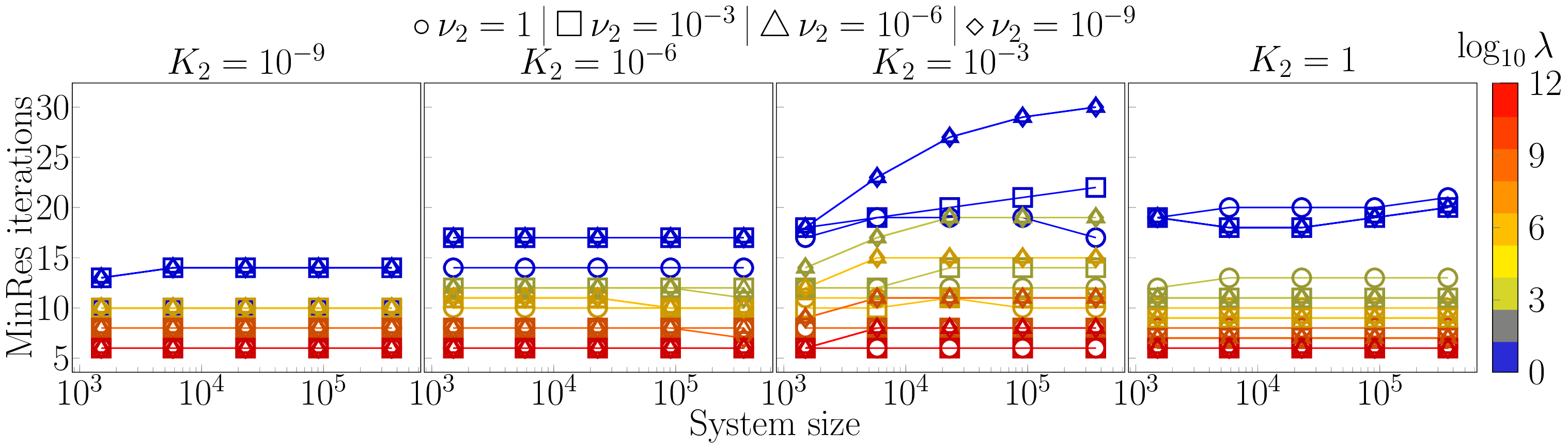}
  \vspace{-20pt}        
  \caption{Number of preconditioned MinRes iterations for 2-network Biot-Brinkman
    system with preconditioner \eqref{Preconditioner:B}.
    (Top) The displacement and flux blocks use realized by geometric multigrid while
    $\mathcal{B}_{\bm p}$ is computed by LU. (Bottom) Exact (LU-inverted) preconditioner
    is used. Transfer coefficient $\beta=10^{6}$, while $c_2=0$, $\alpha_2=1$ and the remaining
    problem parameters are set to 1.}
  \label{fig:flux_brinkman_beta1E6}
\end{figure}

\begin{figure}
  \centering
  \includegraphics[width=\textwidth]{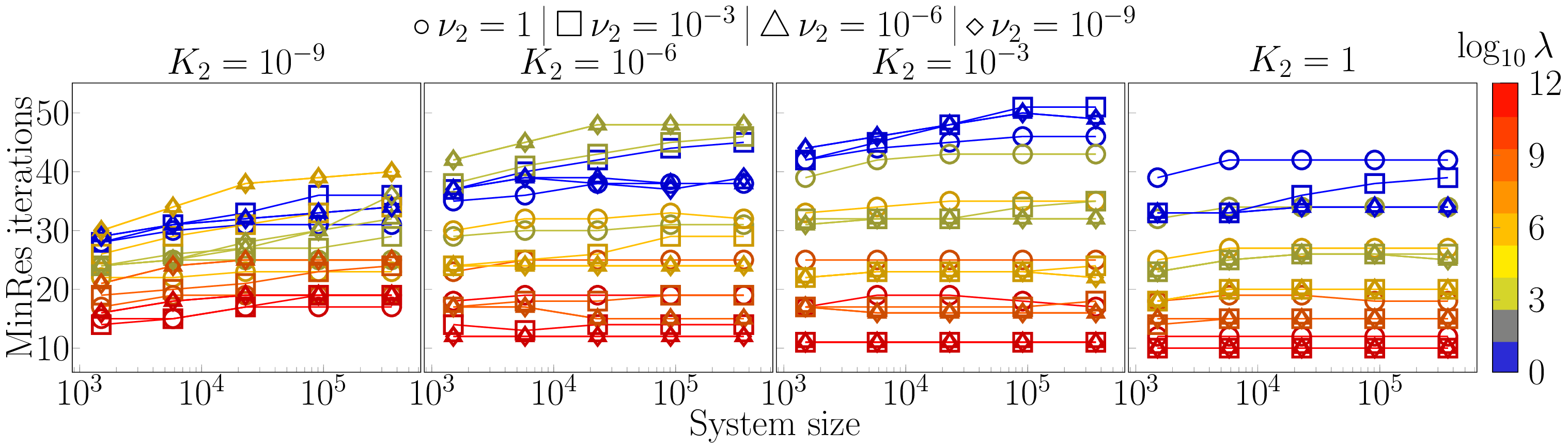}\\
  \includegraphics[width=\textwidth]{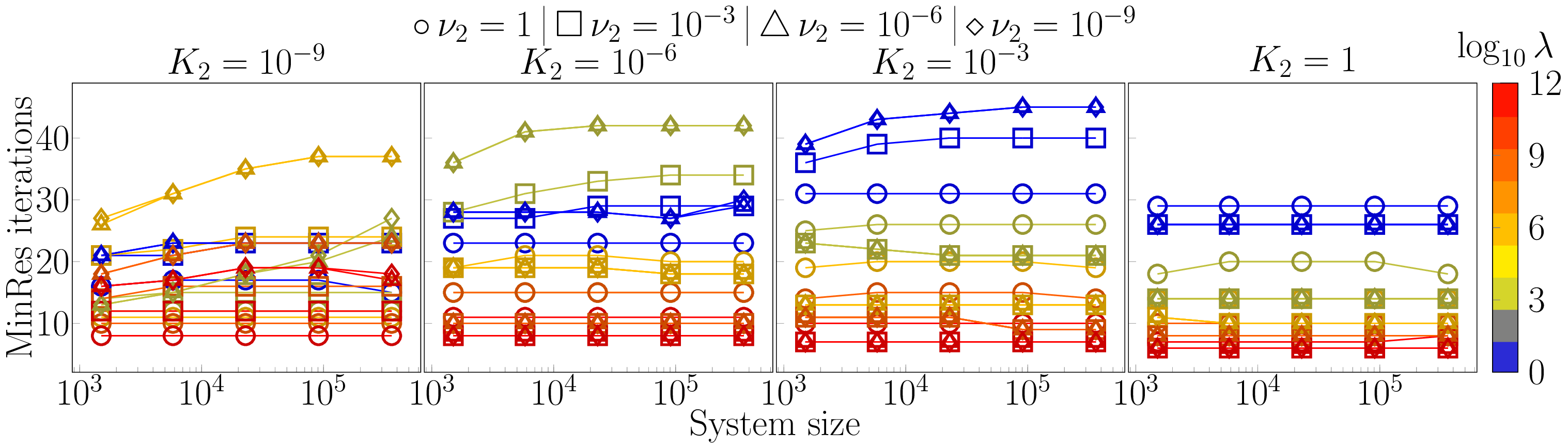}
  \vspace{-20pt}      
  \caption{
    Number of preconditioned MinRes iterations for 2-network Biot-Brinkman
    system with preconditioner \eqref{Preconditioner:B}.
    (Top) The displacement and flux blocks use realized by geometric multigrid while
    $\mathcal{B}_{\bm p}$ is computed by LU. (Bottom) Exact preconditioner
    is used. Transfer coefficient $\beta=10^{-6}$, while $c_2=0$, $\alpha_2=1$ and the remaining
    problem parameters are set to 1.}
  \label{fig:flux_brinkman_beta1E-6}
\end{figure}

We finally compare the cost of the exact and inexact Biot-Brinkman preconditioners
for case $K_2=10^{-3}$, $\lambda=1$, $\beta=10^{-6}$ which required most iterations
in the previous experiments, cf. \cref{fig:flux_brinkman_beta1E-6}. Our results
are summarized in \cref{tab:mg_vs_lu}. We observe that despite requiring more
iterations for convergence the solution time\footnote{The comparison is done in terms
  of the aggregate of the setup time of the preconditioner and the run time of the Krylov solver.
} with the multigrid-based preconditioner is noticeably faster. In addition,
the resulting solution algorithm appears to scale linearly in the number of
unknowns. We remark that for the sake of simple comparison the computations
were done in serial using single-threaded execution. However, the latter setting
is particularly unfavorable for the exact preconditioner $\mathcal{B}$ as
modern LU solvers are known for their thread efficiency.

\begin{table}
  \centering
  \footnotesize{
  \begin{tabular}{c|cccc||cccc}
    \hline
     & \multicolumn{4}{c||}{MinRes iterations LU} & \multicolumn{4}{c}{MinRes iterations MG}\\
    \hline
\backslashbox{$\nu_2$}{$h$} & $2^{-3}$ & $2^{-4}$ & $2^{-5}$ & $2^{-6}$ & $2^{-3}$ & $2^{-4}$ & $2^{-5}$ & $2^{-6}$\\
    \hline
$10^{-9}$  & 43 & 44 & 45 & 45  & 46 & 48 & 50 & 49\\
$10^{-6}$  & 43 & 44 & 45 & 45  & 46 & 48 & 50 & 49\\
$10^{-3}$  & 39 & 40 & 40 & 40  & 45 & 48 & 51 & 51\\
$1$       & 31 & 31 & 31 & 31  & 44 & 45 & 46 & 46\\
    \hline
    \hline
     & \multicolumn{4}{c||}{Solve time LU [s]} & \multicolumn{4}{c}{Solve time MG [s]}\\
    \hline
\backslashbox{$\nu_2$}{$h$} & $2^{-3}$ & $2^{-4}$ & $2^{-5}$ & $2^{-6}$ & $2^{-3}$ & $2^{-4}$ & $2^{-5}$ & $2^{-6}$\\
    \hline
$10^{-9}$  & 2.50 & 4.64 & 23.18 & 181.00 & 4.47 & 7.05 & 18.22 & 64.29\\  
$10^{-6}$  & 2.51 & 4.65 & 23.12 & 180.36 & 4.59 & 7.15 & 18.21 & 64.47\\  
$10^{-3}$  & 2.50 & 4.64 & 23.06 & 180.34 & 4.57 & 7.05 & 18.24 & 65.45\\  
$1$       & 2.51 & 4.57 & 22.74 & 178.84 & 4.45 & 6.94 & 17.63 & 62.83\\  
    \hline
  \end{tabular}
  }
  \vspace{-5pt}    
  \caption{
    Performance of exact (LU) and approximate multigrid-based (MG) preconditioners for the two-network
    generalized Biot-Brinkman model. Parameter $\nu_2$ is varied while
    $c_2=0$, $K_2=10^{-3}$, $\beta=10^{-6}$ and the remaining parameters are set to 1.
    Number of unknowns in the systems ranges from $6\times 10^3$ to $362\times 10^3$.
    Solve time aggregates setup time of the preconditioner and the run time of the Krylov solver.
    Computations were done in serial with threading disabled by setting \texttt{OMP\_NUM\_THREADS=1}.
  }
  \label{tab:mg_vs_lu}
\end{table}

% ---------------------------------------------------------------------

\bibliographystyle{siamplain}
\bibliography{reference_mpet}
\appendix
\section{Components of multigrid preconditioner}\label{sec:mg_components}
In this section we report numerical experiments demonstrating robustness
of geometric multigrid preconditioners for blocks $\mathcal{B}_{\bm u}$ and
$\mathcal{B}_{\bm v}$ of the Biot-Brinkman preconditioner \eqref{Preconditioner:B}. Adapting the unit square geometry
and the setup of boundary conditions from \Cref{sec:mg} we investigate performance
of the preconditioners by considering boundedness of the (preconditioned) conjugate gradient (CG) 
iterations. In the following, the initial vector is set to 0 and the convergence of the CG
solver is determined by reduction of the preconditioned residual norm by a factor $10^{8}$.
Finally, both systems are discretized by  $\text{BDM}_1$ elements.

\cref{tab:elasticity_mg} confirms robustness of the $F(2, 2)$-cycle for
the displacement block of \eqref{Preconditioner:B}. In particular, the
iterations can be seen to be bounded in mesh size and the Lam{\' e} parameter $\lambda$.

\begin{table}
  \centering
  \scriptsize{
    \begin{tabular}{c|cccccc}
      \hline
      \multirow{2}{*}{$\lambda$} & \multicolumn{6}{c}{$\log_2h$}\\
      \cline{2-7}
      & ${-3}$ & ${-4}$ & ${-5}$ & ${-6}$ & ${-7}$ & ${-8}$ \\
      \hline
      1        & 10 & 10 & 9 & 9 & 9 & 9\\
      $10^{3}$  & 14 & 14 & 13 & 13 & 12 & 12\\
      $10^{6}$  & 14 & 14 & 13 & 13 & 13 & 12\\
      $10^{9}$  & 14 & 14 & 14 & 13 & 13 & 13\\
      $10^{12}$ & 14 & 15 & 14 & 14 & 15 & 16\\
      \hline
    \end{tabular}
  }
  \vspace{-5pt}  
  \caption{
    Number of preconditioned conjugate gradient iterations for approximating 
    the displacement block $\mathcal{B}_{\bm u}$ of the Biot-Brinkman preconditioner.
    Geometric multigrid preconditioner uses $F(2, 2)$-cycle with 3 levels and a vertex-star (damped Richardson)
    smoother. In all experiments $\mu=1$.
  }
  \label{tab:elasticity_mg}
\end{table}

For the flux block $\mathcal{B}_{\bm v}$ we limit the investigations to the
two-network case and set $c_2=0$, $\alpha_2=1$ as these parameter values yielded
the stiffest problems (in terms of their condition numbers) in the robustness study
of \Cref{sec:sensitivity}. Performance of the geometric multigrid 
preconditioner using a $W(2, 2)$-cycle with vertex-star smoother is then summarized
in \Cref{fig:flux_mg}. We observe that the number of CG iterations is bounded
in the mesh size and variations in $K_2$, $\nu_2$ and the exchange coefficient $\beta$.

\begin{figure}
  \centering
  \includegraphics[width=\textwidth]{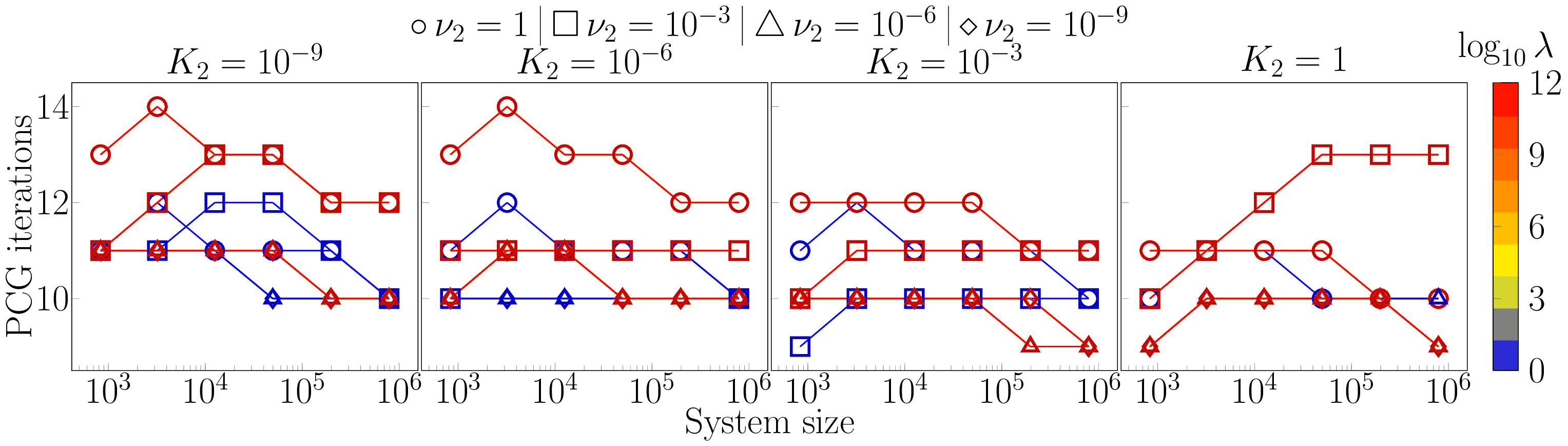}\\
  \includegraphics[width=\textwidth]{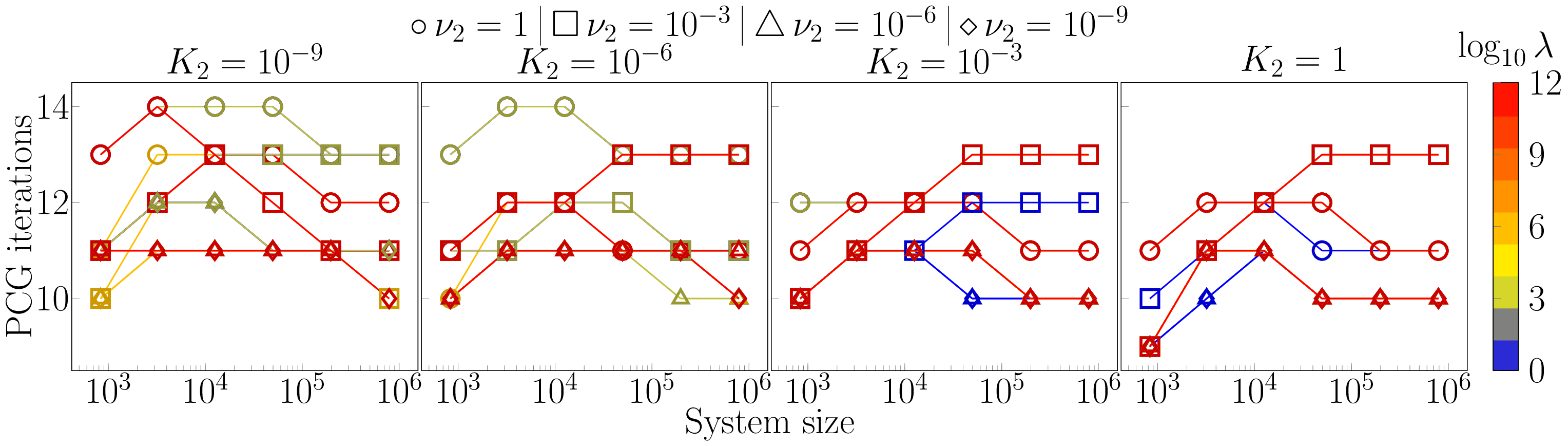}
  \vspace{-20pt}
  \caption{Number of preconditioned conjugate gradient iterations for approximating
    the flux block $\mathcal{B}_{\bm v}$ of the Biot-Brinkman preconditioner.
    The preconditioner uses $W(2, 2)$-cycle of geometric multigrid with vertex-star (damped Richardson)
    smoother and 3 grid levels. (Top) Transfer coefficient $\beta=10^{6}$,
    (bottom) $\beta=10^{-6}$. Values of $K_2$, $\nu_2$ (encoded by markers)
    and $\lambda$ (encoded by line color) are varied. In both setups $c_2=0$, $\alpha_2=1$ and the remaining
    problem parameters are set to 1.
  }
  \label{fig:flux_mg}
\end{figure}

\end{document}